\documentclass[oneside,leqno,12pt]{article}

\usepackage{graphicx}
\usepackage{xcolor}
\usepackage{tikz}
\usepackage{amssymb}
\usepackage{amsmath}
\usepackage{nicematrix}
\usepackage{amsfonts}
\usepackage{amsthm}
\usepackage{comment}
\usepackage{mathtools}
\usepackage{centernot}
\usepackage{tensor}
\usepackage[backref=page]{hyperref}
\usepackage{fancyhdr}
\usepackage{pxfonts}
\usepackage[T1]{fontenc}
\usepackage[toc,title]{appendix}
\usepackage[noabbrev]{cleveref}
\usepackage[inline]{enumitem}
\usepackage{import}
\usepackage[font=footnotesize]{caption}
\usepackage[normalem]{ulem}
\usepackage{setspace}
\usepackage{pdfpages}

\crefformat{section}{\S#2#1#3}
\crefformat{subsection}{\S#2#1#3}

\usepackage[margin=2cm]{geometry}

\numberwithin{equation}{section}
\newtheorem{theorem}{Theorem}[section]

\newtheorem{lemma}[theorem]{Lemma}

\theoremstyle{definition}
\newtheorem{definition}[theorem]{Definition}

\newtheorem{remark}[theorem]{Remark}

\renewcommand{\O}{\Omega}
\renewcommand{\o}{\omega}
\renewcommand{\a}{\alpha}

\newcommand{\g}{\gamma}

\newcommand{\eps}{\varepsilon}

\newcommand{\sym}{\text{sym}}

\newcommand{\D}{\nabla}
\newcommand{\II}{\mathrm{I\!I}}

\renewcommand{\P}{\mathbb{P}}
\newcommand{\R}{\mathbb{R}}

\newcommand{\p}{\partial}
\newcommand{\oII}{\overline{\II}}

\newcommand{\og}{\overline{g}}
\newcommand{\oS}{\overline{S}}

\newcommand{\K}{\widetilde{K}}
\newcommand{\wtA}{\widetilde{\mathcal{A}}}
\newcommand{\A}{\mathcal{A}}
\newcommand{\E}{\mathcal{E}}

\DeclareMathOperator{\dist}{dist}

\DeclareMathOperator{\SO}{SO}

\DeclareMathOperator{\Sym}{Sym}
\DeclareMathOperator{\Tr}{Tr}
\DeclareMathOperator*{\Glim}{\Gamma-lim}
\DeclareMathOperator{\id}{Id}

\newcommand{\xto}[1]{\xrightarrow[]{#1}}

\newcommand{\wto}{\rightharpoonup}
\newcommand{\xwto}[1]{\xrightharpoonup[]{#1}}
\newcommand{\imm}{\quad\Leftrightarrow\quad}

\newcommand{\norm}[1]{\left\lVert#1\right\rVert}     
\newcommand{\inn}[1]{\left\langle#1\right\rangle}    
\newcommand{\abs}[1]{\left|#1\right|}                
\newcommand{\bra}[1]{\left(#1\right)}                
\newcommand{\sqbra}[1]{\left[#1\right]}              
\newcommand{\set}[1]{\left\{#1\right\}}              

\makeatletter
\newcommand{\superimpose}[2]{{%
		\ooalign{%
			\hfil$\m@th#1\@firstoftwo#2$\hfil\cr
			\hfil$\m@th#1\@secondoftwo#2$\hfil\cr
		}%
}}
\makeatother
\newcommand{\fII}{
	\mathpalette\superimpose{{\rotatebox{12}{\textendash}}{\II}}
}

\newcommand{\wtfII}{
	\mathpalette\superimpose{{\rotatebox{12}{\textendash}}{\overline{\II}}}
}

\title{Shape-Transition in Non-Euclidean Ribbons}
\author{Noam Shalev \\
	\normalsize Advisor: Cy Maor}
\date{}

\begin{document}

\maketitle
\begin{center}
	{\Large\bfseries Abstract}
\end{center}
\vspace{0.1cm}
\begin{center}
\begin{minipage}{0.8\textwidth}
\setstretch{1.4} 

Ribbons are thin elastic bodies whose thickness $t$ is much smaller than their width $w$, which is in turn much smaller than their length. 
Starting from a three-dimensional model, 
we derive a one-dimensional limit theory for ribbons with rough prestrain, by means of $\Gamma$-convergence. 
Our model shows that for narrow ribbons, the energy minimizer coincides with the reference (effective) second fundamental form along the midline, while for wide ribbons this is generically not the case. 
This proves the existence of shape-transitions in such ribbons, as observed in many experiments, generalizing recent results by Maor \& Mora to rough prestrains, which more accurately model many of the relevant physical systems.
Our analysis combines techniques from the study of Euclidean ribbons due to Freddi et al., the work of Schmidt on dimension reduction of prestrained plates, and a new structure theorem on the scaled limiting strain due to Maor \& Mora.

\end{minipage}
\end{center}
\vspace{1cm}
\section*{Acknowledgements}
I am deeply grateful to my advisor, Prof.~Cy Maor, for his patient guidance, intellectual encouragement, and clarity in explaining even the most challenging ideas. His mentorship has been invaluable to both this work and my growth as a researcher.

I thank my wife, Zohar, for her love and support throughout this journey, and my family for always being there for me.
\newpage
\tableofcontents
\section*{Reader's Guide}

This thesis assumes familiarity with the mathematical theory of (non-linear) elasticity 
and the framework of $\Gamma$-convergence.
Readers well-versed in these topics may skip Section~\ref{sec:background} entirely.
Readers less familiar with this material are encouraged to begin with 
Section~\ref{sec:background}, which provides a self-contained overview, 
before continuing to the main body of the thesis. For a more thorough 
treatment, we recommend \cite{antman95, ciarlet00, lewicka-prestressed}.

The main results are stated in Section~\ref{sec:introduction}. 
The principal contributions of this thesis are contained in 
Sections~\ref{sec:narrow-ribbons} and~\ref{sec:wide-ribbons}.

\newpage
\section{Introduction}\label{sec:introduction}
\subsection{Prestrained Ribbons and Shape Transitions}\label{ss:prestrained-ribbons}
Ribbons are thin elastic bodies whose thickness $t$ is significantly smaller than their width $w$, i.e., $t\ll w \ll L$ (where $L$ is the length of the midline of the ribbon). Along with other thin bodies such as plates or rods, ribbons have attracted significant mathematical investigation; we refer to \cite{fhmp16b,mm25} for a thorough overview of relevant literature and applications.  

Many ribbons have internal geometries
that are incompatible with Euclidean space, e.g., due to inhomogeneous swelling, plastic deformations
or differential growth. Such ribbons do not have a stress-free configuration: they exhibit stress even
in the absence of external forces or prescribed boundary conditions.
One central goal of the study of ribbons is to determine
the relation between their local intrinsic geometry and their emergent global structure. 
We can view ribbons as an interpolation between plates ($w \approx L$) and rods $(w\approx t$). Loosely speaking, wider ribbons are more plate-like, so we expect that their emergent configuration is less flexible, as their second fundamental form is constrained by the metric of their midsurface (via the Gauss--Codazzi equations). 

Narrower ribbons, however, are expected to behave like rods, and their  second fundamental form is not constrained. 
This \emph{shape-transition} has been observed in recent experiments \cite{gsd16,zgds19, sls21,lssm21} for certain geometries; it was explained by formal asymptotics in \cite{gsd16, lssm21}, and a rigorous analysis of this phenomenon was first given recently in \cite{mm25}.  

In this last work, the transition is exhibited based on the relationship between the thin dimensions of the ribbon: 
different limiting 1D-theories are obtained when taking $t\to0$ and then $w\to 0$ (\textit{wide} ribbons), versus taking the double limit $t,w\to 0$ under the assumption $w^2 \ll t$ (\textit{narrow} ribbons).\footnote{The problem of deriving other limits (such as double limits with $w^2 \gtrsim t$) from three-dimensional elasticity is still open.}
Analyzing the behavior of minimizers of these limiting 1D-theories for a given geometry shows the existence of a shape transition.
 
 The starting point --- i.e., the three dimension ribbon model --- considered in \cite{mm25} and in this work is as follows:
 The domain of the ribbon is given by 
$\O_{t,w} = (0,L)\times (-\frac{w}{2},\frac{w}{2})\times (-\frac{t}{2},\frac{t}{2})$, and the elastic energy associated with a configuration $v\in W^{1,2}(\O_{t,w};\R^3)$ is
	\[
	E^{t,w}(v)= \fint_{\O_{t,w}} W(\D v\circ P^{-1}_{t,w})\,dV,
	\]
where $W:\R^{3\times 3}\to [0,\infty]$ is the elastic energy density (satisfying standard assumptions detailed below),  $P_{t,w}:T\O_{t,w}\to \R^3$ is
an orientation-preserving invertible map (representing the prestrain) which induces a (rough) metric on $\O_{t,w}$ via $g_{t,w}(u_1,u_2)=\inn{P_{t,w}(u_1),P_{t,w}(u_2)}$, and $V$ is an appropriate volume form. 
Here, $\fint$ means the integral divided by the volume of the domain.

When the prestrain is non-flat (i.e., when $g_{t,w}$ is smooth enough and its  Riemannian curvature tensor does not vanish),  the elastic energy of any configuration is bounded away from zero \cite{lewpak11},
meaning that $\eps_{t,w}^2\coloneqq \inf E^{t,w} >0$.  
We refer to $\eps_{t,w}$ (which is well-defined also for rough prestrains) as the \emph{natural energy scaling} of the problem; whenever it does not vanish, we would like to obtain a limiting theory (in terms of $\Gamma$-convergence) in this scaling.

The main difference between this work and \cite{mm25} is the generality of the map $P_{t,w}$ considered.
Namely, in \cite{mm25} it is assumed that $P_{t,w}$ is a restriction to $\O_{t,w}$ of a smooth map $P:\O_{1,1}\to\R^3$. 
In this work,
we develop limit models for non-Euclidean ribbons whose prestrain has the form 
\begin{equation}\label{eq:prestrain}
P^{-1}_{t,w} = I + t\; B\bra{z_1,\frac{z_2}{w},\frac{z_3}{t}} +o(t)
\end{equation} in local coordinates, for some symmetric matrix-field $B\in L^\infty(\O_{1,1};\R^{3\times 3}_\sym)$ (where here, by $o(t)$ we mean uniformly with respect to $z\in \O_{1,1}$). 

In this notation, the assumptions of \cite{mm25} imply that the field $B$ in \eqref{eq:prestrain} must be independent of $x_2$ and depend linearly on $x_3$.\footnote{We note that the results of \cite{mm25} allow for a more general structure for the terms in \eqref{eq:prestrain} which dominate the $O(t)$ scaling (which are here assumed to be the zeroth-order term $I$).  The results of this thesis generalize to  the same structure as in \cite{mm25}, and the choice of the structure \eqref{eq:prestrain} of the prestrain is in order to focus on the most essential parts of the analysis.}
These assumptions --- while being very natural from a differential-geometric point of view --- fail to encompass the geometry of many of the prestrained/non-Euclidean ribbons studied in the physical literature, which have a multi-layer or other piecewise-constant structure \cite{aeks11,cmsh12, sls21, lssm21}.
These systems typically  have a prestrain structure of the type \eqref{eq:prestrain}, see Figure~\ref{fig:1}.
The aim of this work is thus to generalize the rigorous analysis of shape transitions in ribbons to such physically-relevant geometries.

\begin{figure}
	\captionsetup{width=0.95\linewidth}
	\centering
	\scalebox{1}{\def\svgwidth{\columnwidth}
\begingroup%
  \makeatletter%
  \providecommand\color[2][]{%
    \errmessage{(Inkscape) Color is used for the text in Inkscape, but the package 'color.sty' is not loaded}%
    \renewcommand\color[2][]{}%
  }%
  \providecommand\transparent[1]{%
    \errmessage{(Inkscape) Transparency is used (non-zero) for the text in Inkscape, but the package 'transparent.sty' is not loaded}%
    \renewcommand\transparent[1]{}%
  }%
  \providecommand\rotatebox[2]{#2}%
  \newcommand*\fsize{\dimexpr\f@size pt\relax}%
  \newcommand*\lineheight[1]{\fontsize{\fsize}{#1\fsize}\selectfont}%
  \ifx\svgwidth\undefined%
    \setlength{\unitlength}{1495.33112774bp}%
    \ifx\svgscale\undefined%
      \relax%
    \else%
      \setlength{\unitlength}{\unitlength * \real{\svgscale}}%
    \fi%
  \else%
    \setlength{\unitlength}{\svgwidth}%
  \fi%
  \global\let\svgwidth\undefined%
  \global\let\svgscale\undefined%
  \makeatother%
  \begin{picture}(1,0.22630034)%
    \lineheight{1}%
    \setlength\tabcolsep{0pt}%
    \put(0,0){\includegraphics[width=\unitlength,page=1]{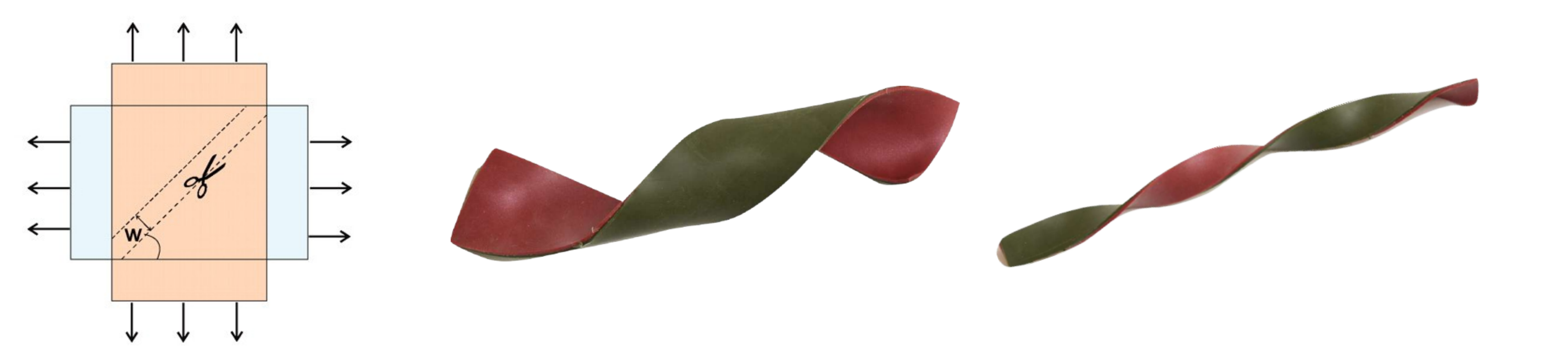}}%
    \put(0.41499429,0.20597668){\color[rgb]{0,0,0}\makebox(0,0)[lt]{\lineheight{1.25}\smash{\begin{tabular}[t]{l}Wide\end{tabular}}}}%
    \put(0.78112927,0.20641423){\color[rgb]{0,0,0}\makebox(0,0)[lt]{\lineheight{1.25}\smash{\begin{tabular}[t]{l}Narrow\end{tabular}}}}%
    \put(0,0){\includegraphics[width=\unitlength,page=2]{paper_fig.pdf}}%
    \put(0.10763914,0.06958769){\color[rgb]{0,0,0}\makebox(0,0)[lt]{\lineheight{1.25}\smash{\begin{tabular}[t]{l}\scalebox{0.7}{$45^\circ$}\end{tabular}}}}%
  \end{picture}%
\endgroup%
}
\caption{
A prototypical example of an experimentally observed shape-transition (figure adapted from \cite{aeks11, lssm21}). Two planar latex sheets are stretched uniaxially along perpendicular directions and then glued together, forming a residually stressed compound sheet. A strip is then cut from this sheet along a direction that forms an angle $45^\circ$ with the stretching direction. The ribbon transitions between
two helical configurations: one cut from a cylinder in the wide regime, and a helicoid in the narrow regime. This scenario may be modelled by a bilayered ribbon whose prestrain has the (non-smooth) piecewise-defined form $P= I+t B_1$ when $x_3>0$ and $P=I+t B_2$ when $x_3<0$, where $B_1\neq B_2$ are distinct constant matrix-fields.
}
\label{fig:1}
\end{figure}

In general, for prestrains of the type \eqref{eq:prestrain}, one should expect that the energy scaling will be $\eps_{t,w}\sim t$ in all regimes, as this is the scaling for both non-Euclidean plates \cite{sch07} and rods \cite{ko18}.
Indeed, part of the results of this work will show that this is the natural energy scaling.
This scaling is of higher energy than the ones observed in \cite{mm25} (in some regimes), and in rods with smooth metrics \cite{ms19}, and occurs  because rough prestrains generically induce some non-zero excess energy, which bars $\inf E^{t,w}$ from being $o(t^2)$.

\subsection{Main Results.}
 We show that the intrinsic energy scaling is indeed $\eps_{t,w}\sim t$ for a generic prestrain (where the precise notion of non-genericity is given by \eqref{eq:non-generic-prestrain}). 
 The limit models we obtain as $t,w\to 0$ show that for narrow ribbons, the energy minimizer represents a surface whose second fundamental form coincides with the reference second fundamental form along the midline, while for wide ribbons it does not, if the reference second fundamental form has non-zero determinant --- that is, if it violates the Gauss equation with respect to the Euclidean metric (in the terminology of \cite{lssm21,mm25}, this is called \emph{Gauss incompatibility}).
 Thus we provide a rigorous proof of shape-transition in this case.
 More specifically, we show:
\begin{enumerate}
\item 
For narrow ribbons ($w^2\ll t$), the $\Gamma$-limit admits the form
\begin{equation}\label{eq:I0-intro}
I^{0}(\fII) = \frac{1}{24}\int_0^L Q_2(\fII-\wtfII)dx_1 + C_{\text{excess}},
\end{equation}
where $\fII$ is the second fundamental form (of the midsurface, along the midline) \emph{associated} with a finite-energy configuration (a notion which will be made clear in \S~\ref{sec:narrow-ribbons}).
Here, $\wtfII$ is the \emph{reference} (effective) second fundamental form and  $C_{\text{excess}}\geq0$ is the \emph{excess energy} (both of which are induced only by $W$ and $B$), and $Q_2$ is a quadratic form emerging from $W$ (see \S~\ref{ss:quadratic-forms}).
Unless $C_{\text{excess}}$ vanishes, this shows that the minimal energy is positive, showing that $\eps_{t,w}\sim t $ is indeed the natural scaling.
Note that $C_{\text{excess}}=0$ only if $B$ has a special structure (as in the case treated in \cite{mm25}), see equation \eqref{eq:non-generic-prestrain} for an exact formula; in particular, it does not vanish for bilayered structures as depicted in Figure~\ref{fig:1}.
Furthermore, as can be immediately seen from equation \eqref{eq:I0-intro}, the energy minimizer is $\fII=\wtfII$.

\item
For wide ribbons (first $t\to 0$, then $w\to 0$), the (iterated) $\Gamma$-limit admits the form
\[	J^0(\fII)\coloneqq
	\frac{1}{24}\int_0^L \bra{
		Q_2(\fII-\wtfII)+\alpha_Q^+(\det \fII)^+ + \alpha_Q^-(\det \fII)^-
	}dx_1 +C_{\text{excess}},
\]
where $\alpha_Q^\pm$ are positive constants induced by $Q_2$. We note that this $\Gamma$-limit is obtained under an additional structure assumption on $B$; see \S~\ref{sec:wide-ribbons} for details. Hence, for generic wide ribbons --- i.e., whenever $C_{\text{excess}}$ is positive --- the natural scaling is still $\eps_{t,w} \sim t$. Moreover, adapting from \cite[Proposition~5.8]{mm25}, we conclude that an energy minimizer $\fII_{\text{min}}$ of $J^0$ satisfies the pointwise bound \[\abs{\fII_{\text{min}}(x_1)-\wtfII(x_1)}\geq c \abs{\kappa_1(x_1)}\]
for some $c>0$, where $\kappa_1$ is the smallest magnitude eigenvalue of $\wtfII$. Hence, in the Gauss-incompatible case (where $\det \wtfII \neq 0$), the limiting second fundamental form \emph{differs} from the reference form, and thus from the minimizer of \eqref{eq:I0-intro}.
\end{enumerate}

We also supply each of the above $\Gamma$-convergence results with the respective compactness result, which is a prerequisite for concluding that approximate minimizers converge to the minimizer of the limiting energy (see \S\ref{sec:gamma} for details).

\subsection{Relation to Previous Works}
The study of dimension-reduction theories for thin elastic bodies has a long history;
we mention \cite{k1850, love27, antman95, ciarlet00} for classical accounts and \cite{ldr96, bff00,fjm2002, mm03,fjm06, mm08, fmp12} for rigorous results in Euclidean elasticity (without attempting to be exhaustive).   
Related limit theories for prestrained thin bodies were derived in \cite{sch07,lewpak11, ks14, fhmp16b, crs17, bls16, lrr17,ko18,all19,ms19,ll20,bnps22, pg22, lewicka-prestressed}.

The case of narrow ribbons in this work builds on \cite{fmp12}, in which  the authors obtain a limit model for Euclidean narrow ribbons ($B\equiv 0$) in the same energy regime $\eps_{t,w}\sim t$. 
There, the limiting energy functional depends only on the first column of the second fundamental form, i.e., the variables that can directly be deduced from a frame along the midline, similar to rod-like limits.
Here we opted for a more detailed limiting functional, involving the full second fundamental form along the midline, for two reasons: first, it is the convention in many works in the physics literature on non-Euclidean/prestrained ribbons \cite{gsd16,sls21,lssm21}, and, more importantly, since the essence of the shape transition we show is the difference of the minimal second fundamental form along the midline between the difference regimes.
Thus, our notion of convergence is somewhat more involved.

The main challenge in treating a general prestrain $B$ (for narrow ribbons) arises from the in-plane-dependence of the prestrain --- indeed, for prestrains of the form $I+t B(x_3/t)$, the result can be obtained by combining the analysis \cite{fmp12} of flat ribbons and the analysis \cite{sch07} of plates with bilayer prestrains.
To treat this general case, the main tool is a more detailed structure theorem on the scaled limiting strain which was discovered recently in \cite{mm25} (see \S~\ref{sec:structure-strain} below).
Our analysis reveals the geometric meaning of this result: 
For configurations with energy $\eps_{t,w} \sim t$, the second fundamental form of the midsurface is of order one, and the geodesic curvature of the midline in the midsurface is of order $t/w$. 
The structural result shows how the limiting strain contains the information of both these contributions. 
Thus, in order to obtain a limiting energy that depends only on the difference between the second fundamental form of the configuration and the reference (effective) one that is induced by prestrain $B$, we must also read off from $B$ the leading order geodesic curvature (even though it vanishes in the limit), and construct the recovery sequence such that it matches this geodesic curvature.

To complete the shape-transition picture, one must establish that in the wide ribbon regime, energy minimizers differ from the reference second fundamental form along the midline. Our analysis in this case largely follows the arguments in \cite{mm25}, adapting the prestrained shell theory developed in \cite{pg22}. A new obstacle arises, however, from the in-plane-dependence of $B$; specifically, the $x_2$-dependence of its $(2,2)$-component. For generic $B$, this dependence prevents us from matching the proposed lower bound with a recovery sequence. We therefore impose a structural assumption on $B$ that circumvents this issue; see \S\ref{sec:wide-ribbons} for details.

\subsection{Outline}
In Section~\ref{sec:background} we provide a self-contained exposition on $\Gamma$-convergence and other concepts in the theory of non-linear elasticity --- readers familiar with these topics may skip that section.
In Section~\ref{sec:setting} we describe the setup of the three-dimensional model under consideration and the corresponding energy functional. In Section~\ref{sec:preliminaries} we gather the relevant preliminaries, including (i) the definitions of the reference second fundamental form and excess energy, (ii) compactness lemmas for narrow ribbons, and (iii) remarks regarding the scaled limiting strain. In Section~\ref{sec:narrow-ribbons} we present the limiting theory for narrow ribbons, complete with rigorous derivations of the lower bound and the recovery sequence. Lastly, in Section~\ref{sec:wide-ribbons}, we rigorously derive a limiting model for wide ribbons --- at this time under an additional assumption on the prestrain.

\newpage
\section{Tools and Background}\label{sec:background}
This section provides a brief overview of the variational 
theory of elastic thin bodies. It is not intended to be comprehensive; 
for a thorough treatment, we refer to \cite{ciarlet00} for classical 
elasticity, \cite{lewicka-prestressed} for modern variational approaches, and the lectures notes \cite{cynotes} for an overview of non-Euclidean elasticity.

\subsection{$\Gamma$-Convergence}\label{sec:gamma}
The notion of \emph{$\mathit{\Gamma}$-convergence} provides a natural framework for studying limits of variational problems. It is central to the dimension reduction analysis in this thesis: when combined with an appropriate compactness assumption, $\Gamma$-convergence guarantees that approximate minimizers of a sequence of energy functionals converge to minimizers of the $\Gamma$-limit. The standard treatment of $\Gamma$-convergence takes place in metric spaces; we refer to \cite{lewicka-prestressed} for a short introduction to the topic. Here, we adopt a more general viewpoint, formulating the theory for any abstract notion of sequential convergence.

To this end, suppose throughout this section that we have some notion of sequential convergence relating sequences $(x_n) \subseteq X$ to limits $y \in Y$, where $X$ and $Y$ may be distinct spaces (formally, this can be viewed as convergence in some ambient space containing both $X$ and $Y$). We assume that the notion of convergence satisfies the following natural properties: (i) limits are unique, and (ii) every subsequence of a convergent sequence converges to the same limit.

\begin{definition} Let $F_n: X \to [-\infty,\infty]$ be a sequence of functions, and let $F: Y\to [-\infty,\infty]$. We say that this sequence \textit{$\mathit{\Gamma}$-converges} to $F$ (and write $F_n \xto{\Gamma} F$ or $\Glim F_n = F$), whenever the two following conditions hold: 
\begin{itemize}
\item[(i)] For every $y\in Y$ and every sequence $(x_n)\subseteq X$, if $x_n\to y$ then $F(y)\leq \liminf_{n\to \infty} F_n(x_n)$.
\item[(ii)] For every $y\in Y$, there exists a sequence $(x_n)\subseteq X$ (called a  \emph{recovery sequence}) such that $x_n\to y$ and $F(y) = \lim_{n\to \infty} F_n(x_n)$.
\end{itemize}
\end{definition}

\begin{theorem} Assume that $F_n \xto{\Gamma} F$ (where $F_n$ are defined on $X$ and $F$ on $Y$). Suppose that the following notion of compactness is satisfied:
\begin{itemize}
	\item If $(x_n)\subseteq X$ is a sequence such that $\limsup_{n\to \infty} \abs{F_n(x_n)} < \infty$, then, up to passing to a subsequence, $x_n\to y$ for some $y\in Y$.
\end{itemize}
Then the limit function $F$ has at least one minimizer in $Y$, and moreover:
\begin{itemize}
	\item[(i)] If $(x_n)\subseteq X$ is a sequence of \emph{approximate minimizers}, i.e., if it satisfies
\begin{equation}\label{eq:approx-min}
\lim_{n\to \infty} \abs{F_n(x_n) - \inf_X F_n} =0,
\end{equation} 
then $(x_n)$ has a subsequence that converges to a minimizer of $F$.
\item[(ii)] For every minimizer $y$ of $F$, there exists a sequence $(x_i)$ of approximate minimizers  such that $x_n\to y$.
\end{itemize}
\end{theorem}

\begin{proof}
Let $(x_n)\subseteq X$ be a sequence of approximate minimizers. Because $F_n\xto{\Gamma} F$,
we have $\limsup_{n\to \infty} \inf_X F_n <\infty$. Combining with \eqref{eq:approx-min}, we conclude that (up to subsequences) $x_n\to y$ for some $y\in Y$. Now let $\bar{y}\in Y$ be arbitrary, and let $(\bar{x}_n)\subseteq X$ be a recovery sequence for $\bar{y}$. Then
\begin{equation}\label{eq:gamma-conv-ineq-chain}
F(y) \leq \liminf_{n\to \infty} F_n(x_n) \leq \liminf_{n\to \infty} \inf_X F_n \leq \liminf_{n\to \infty} F(\bar{x}_n) = F(\bar{y}).
\end{equation}
Therefore $y$ is a minimizer of $F$, and we also proved (i). Additionally, taking $\bar{y}=y$ in \eqref{eq:gamma-conv-ineq-chain} and using the definition of $\Gamma$-convergence, we obtain
$\inf_Y F = F(y) = \lim_{n\to \infty} \inf_X F_n$. For (ii), let $y$ be any minimizer of $F$ and let $(x_n)$ be a recovery sequence for $y$. Then, since $\lim_{n\to \infty} F_n(x_n)=F(y)=\inf_Y F$, we conclude that \eqref{eq:approx-min} holds, hence $(x_n)$ are approximate minimizers.
\end{proof}

\subsection{Elasticity Theory}\label{sec:elasticity}
\subsubsection{Hyperelastic Materials}\label{ss:hyperelastic}
Let $\O\subseteq \mathbb{R}^3$ be an open, bounded, connected domain with Lipschitz boundary, which we will refer to as an \textit{elastic body}. 
A function $y:\O \to \mathbb{R}^3$ is called a \textit{configuration} or \textit{deformation} of $\O$ (we will shortly specify the regularity assumptions on $y$). 
In the framework of a homogenous hyperelastic material, the \textit{elastic energy} corresponding to a configuration $y$ is given by 
\[E(y) = \int_\O W(\D y)dx,\]
where the \textit{elastic energy density} $W:\mathbb{R}^{3\times 3} \to [0,\infty]$ is assumed to satisfy:
\begin{itemize}
	\item[(i)] \textit{Regularity:} $W$ is $C^2$ in a neighbourhood of $\SO(3)$. 
	\item[(ii)] \textit{Energy well:} $W(F)=0$ if and only if $F\in \SO(3)$.
	\item[(iii)] \textit{Coercivity:} $W(F)\geq C \dist^2(F,\SO(3))$.
	\item[(iv)] \textit{Frame indifference:} $W(RF)=W(F)$ for all $R\in \SO(3)$. 
\end{itemize}
Intuitively, the energy density $W$ measures, at each point, how far 
the local deformation is from being a rigid motion. 
As we can see, the elastic energy depends on the derivative $\D y$ of the configuration, often referred to as the \textit{deformation gradient}. The natural setting for configurations in our case will be $W^{1,2}(\O; \R^3)$, which accommodates the required differentiability. Thus the elastic energy is a functional $E:W^{1,2}(\O;\R^3)\to [0,\infty]$.

\subsubsection{Thin Elastic Bodies}\label{ss:thin-bodies}
A thin body in Euclidean space is a three dimensional rectangular cuboid in which one or more of the sides are very small, relative to the others. 
We may classify thin films according to the relation between their dimensions: if 
$ \O_{t,w} = (0,L)\times \left(-\frac{w}{2},\frac{w}{2}\right)\times \left(-\frac{t}{2},\frac{t}{2}\right)$ has width $w$ and height $t$, then $\O$ is a thin body of one of the following types:
\begin{itemize}
	\item \textit{Plates:} $w\sim 1$ and $t\ll L$
	\item \textit{Rods: } $w\sim t \ll L$
	\item \textit{Ribbons:} $t\ll w\ll L$.
\end{itemize}
We seek to understand the limiting behaviour (in the sense of $\mathit{\Gamma}$-\textit{convergence}) of the rescaled energies
\[\widetilde{E}_{t,w}:W^{1,2}(\O;\R^3)\to [0,\infty], \hspace{5em}
\widetilde{E}_{t,w}(y) = \frac{1}{\eps_t^2} \fint_{\O_{t,w}} W(\D y) dx,\]
where $\eps_t^2$ is the typical energy. 
In this Euclidean case, $\eps_t$ is not intrinsic to the problem, but arises from boundary conditions or external forces. Indeed, if we do not impose any constraints, then for the identity configuration $y(x)=x$ we have vanishing energy regardless of $t,w$. Thus the question is: given $\eps_t$, what is the limiting problem as $t\to 0$ (and $w\to0$ in the cases of rods or ribbons)?
Different choices of $\eps_t$ give rise to well-studied energy regimes:
\begin{itemize}
	\item \textit{Kirchhoff bending:} $\eps_t=t^2$
	\item \textit{Linearised Kirchhoff:} $t^4\ll \eps_t \ll t^2$
	\item \textit{Föppl–von Kármán theory:} $\eps_t=t^4$
	\item \textit{Linearised Föppl–von Kármán:} $\eps_t\ll t^4$
\end{itemize}
This thesis will focus on the Kirchhoff bending regime. The case of plates in this regime was studied in the seminal paper \cite{fjm2002}, which also introduced a very useful rigidity estimate. It was shown that with $\eps_t\sim t^2$,
\[\widetilde{E}_t \xrightarrow[]{\quad{\Gamma}\quad} E_0(y) \coloneqq \begin{cases*}
\int_S Q_2(\II_y) dx &$y\in W^{2,2}(S;\mathbb{R}^3)$ and $\D y\in O(2,3)$ a.e.\\
+\infty&otherwise.
\end{cases*}\]
where $Q_2$ is some quadratic form emerging from $W$ (see \S~\ref{ss:quadratic-forms} below), and $\II_y=\D y^T \D n$ is the second fundamental form of $y(S)$, where $S=(0,\ell)\times \left(-\frac{w}{2},\frac{w}{2}\right)$ is called the \textit{midsurface}.
Loosely speaking, $\II_y$ encodes how the normal to $y(S)$ changes (also known as \textit{bending}).

\subsubsection{Non-Euclidean Elasticity}\label{ss:non-euclidean}
Many elastic bodies have internal geometries that are incompatible with Euclidean space (see \S~\ref{ss:prestrained-ribbons} for more details). We model such bodies by means of Riemannian geometry. 
The energy density $W$ above defines the archetypical behaviour of a single point.
In order to write the elastic energy of an elastic body $M$, we need to indicate how to ``implant'' $W$ at each point in the body.\footnote{In rational mechanics, what we describe here is known as a materially-uniform hyperelastic body; see \cite[Section 2]{ekm20} and \cite[Section 2]{Maor25} for recent accounts on this.}
To this end, we define an \emph{elastic body} as a triplet $(M,g,P)$, where $M$ is a compact Riemannian manifold (typically a bounded domain in $\R^3$), $g$ is a Riemannian metric (called the \emph{reference metric}), and $P:TM\to \mathbb{R}^3$ is the \emph{prestrain map} (or implant map, or plastic strain), 
which is an orientation-preserving map that is compatible with $g$: 
for all $v,w\in T_xM$,
\[
g_x(v,w) = \inn{P_x(v),P_x(w)},
\]
from which it follows that $g$ is determined by $P$. 
A \textit{configuration} of $M$ is an orientation-preserving map $v\in W^{1,2}(M;\mathbb{R}^3)$, and the \textit{elastic energy} of a configuration $v:M\to \mathbb{R}^3$ is given by
\[
E_M(v)=E_{M,P}(v) = \fint_M W(Dv\circ P^{-1})dV_{g},
\]
where $dV_{g}$ is the volume form on $M$, and $\fint$ is the integral divided by the total volume.
When $M$ is a domain in $\R^3$, we have a global coordinate chart, and then $dV_{g}=|\det g_x|^{1/2}dx$, and a possible choice of $P$ that is compatible with $g$ is $P_x=g_x^{1/2}$.

Note that by property (ii) of the energy density $W$, a configuration $v$ has vanishing energy if and only if $\D v^T \D v =g$ and $\det \D v >0$, or equivalently if and only if $v$ is an orientation-perserving isometric immersion of $(M,g)$. If $g$ is $C^2$ and $M$ is simply-connected, we have (\cite{lewpak11}):
\begin{theorem} $\inf E_g =0 \imm \mathcal{R}^g\equiv 0 \imm \text{up to changing coordinates,}\; g=I$.
\end{theorem}
On thin bodies $\O_{t,w}=(0, L)\times \left(-\frac{w}{2},\frac{w}{2}\right)\times \left(-\frac{t}{2},\frac{t}{2}\right)$, in which the metric $g=g_t$ may possibly depend on $t$ as well, the rescaled energies are
\[E_{t,w}(y) = \fint_{\O_{t,w}} W(\D y \cdot g^{-1/2}_t)dV_{g_t}.\]
Unlike the Euclidean case, $g_t$ determines the energy scaling $\eps_t^2 =\inf E_{t,w}$, so it is \textit{intrinsic} to the problem.

\subsubsection{Elements of Surface Theory}\label{ss:surface-theory}
Let $r:\o\subseteq \mathbb{R}^2\to \mathcal{S}\subseteq \mathbb{R}^3$ be a coordinate map from a coordinate domain $\o$ onto a surface $\mathcal{S}$ in $\mathbb{R}^3$. Denote by $n$ the unit outward normal to $\mathcal{S}$. The \textit{first fundamental form} $\og_{ij} =\p_ir\cdot\p_j r$ measures distances on $\mathcal{S}$, and the \textit{second fundamental form} $\oII_{ij}=-\p_i r \cdot \p_j n$ measures how the normal changes on $S$ (\textit{bending}). $\og$ and $\oII$ are $(2,0)$-tensors on the tangent bundle $TS$, and they are related by the Gauss-Codazzi equations (G-C):
\begin{itemize}
	\item \textbf{Gauss:} The Gaussian curvature of $S$ is given by  $\overline{K} = \frac{\det \oII}{\det \og}$
	\item \textbf{Codazzi:} $\p_2\oII_{1k}-\p_1\oII_{2k}=\oII_{11}\Gamma_{k2}^1 +\oII_{12}(\Gamma^2_{k2}-\Gamma^1_{k1}) -\oII_{22} \Gamma_{k1}^2$.
\end{itemize}
The first and second fundamental forms characterize $S$ uniquely up to rigid motion:
\begin{theorem}
	If $\og$ and $\oII$ satisfy G-C and $\o$ is simply connected, then there exists a unique immersed surface (up to a rigid motion) $S$ with first fundamental form $\og$ and second fundamental form $\oII$.
\end{theorem}
\begin{definition}
	The \textit{shape operator} $\oS$ is the $(1,1)$-tensor attained from $\oII$ by raising an index:  $\og_{ik} \tensor{\oS}{^k_j} =\oII_{ij}$. It is characterized by 
	\[\og(\oS(v),w) = \oII(v,w),\]
	and it can be easily verified that $\oS(v)=-\D_v n$. 
\end{definition}

\begin{lemma} Let $(\O,g)$ be a Riemannian 3-manifold, and assume that $g$ is given locally w.r.t. a coordinate system $(x_1,x_2,x_3)$ by 
	\begin{equation}
	g(x',x_3)=\begin{pNiceArray}{cc|c}
	\Block{2-2}{\overline{g}(x')}& &0 \\ & &0\\\hline 0&0&1
	\end{pNiceArray}
	- 2 x_3\; \begin{pNiceArray}{cc|c}
	\Block{2-2}{A(x')}& &0 \\ & &0\\\hline 0&0&0
	\end{pNiceArray} + O(x_3^2),
	\label{metricdef}
	\end{equation}
	where $\overline{g},A$ are smooth fields of symmetric positive-definite $2\times 2$ matrices.	Then  $A$ is the second fundamental form of $\mathcal{S}=\{x_3=0\}\subseteq \O$, and $\og$ is the induced metric on $\mathcal{S}$.
\end{lemma}	
\begin{proof} Consider the coordinate vector fields $(\p_1,\p_2,\p_3)$ corresponding to the given coordinates, which form a local frame for $T\O$.
	There are several things which we may readily read off from \eqref{metricdef}:
	\begin{itemize}
		\item[(i)] On $\mathcal{S}$, the vector field $\p_3$ is orthogonal (w.r.t. $g$) to both $\p_1$ and $\p_2$, and furthermore $|\p_3|_g=1$.
		\item[(ii)] Since $(x^1,x^2)$ is a coordinate system for $\mathcal{S}$ (by the definition of $\mathcal{S}$), the vector fields $\p_1$ and $\p_2$ are tangent to $\mathcal{S}$. In particular, in view of (i), the vector field $\p_3$ is the unit normal to $\mathcal{S}$.
		\item[(iii)] Restricting $g$ to $S$, we see that the induced metric on $\mathcal{S}$ is precisely $\overline{g}$.
	\end{itemize}
	Now, by definition, the second fundamental form of $\mathcal{S}$ is given
	by
	\[\II_{ij}=\II\left(\p_i,\p_j\right) = \inn{n, \D^g_{\p_i} \p_j}_g,\]
	where $n$ is the unit normal to $\mathcal{S}$ and $\D^g$ is the Levi-Civita connection on $(\O,g)$.
	  From (ii) we know that $n= \p_3$. Also, $\nabla^g_{\p_i}\p_j=\sum_{k=1}^3 \Gamma_{ij}^k \p_k$, where the Christoffel symbols $\Gamma_{ij}^k$ satisfy (on $\mathcal{S}$)
	\begin{equation}\Gamma_{ij}^k = \sum_{l=1}^3\frac12 g^{kl} \big(\p_i g_{jl}+\p_j g_{il} -\p_l g_{ij}\big).
	\label{christoffel}\end{equation}
	Since $\p_3\;\bot\; \p_1,\p_2$ on $\mathcal{S}$, and since $|\p_3|=1$ on $\mathcal{S}$, we have $\II_{ij}= \inn{\p_3,\Gamma_{ij}^3 \p_3}=\Gamma_{ij}^3$. We shall compute this Christoffel symbol using formula \eqref{christoffel}. Inverting $g$, we have
	\[
	g^{-1}(x',x_3)=
	\begin{pNiceArray}{cc|c}[margin,columns-width=0.3cm]
	\Block{2-2}{\overline{g}^{-1}(x')\,}& &0 \\ & &0\\\hline 0&0&1
	\end{pNiceArray}
	+  x_3 \begin{pNiceArray}{cc|c}[margin,columns-width=0.3cm]
	\Block{2-2}<\Large>{\ast}& &0 \\ & &0\\\hline 0&0&0
	\end{pNiceArray} + O(x_3^2),
	\]
	from which we deduce that $g^{31}=g^{32}=0$ and $g^{33}=1$.	Hence $\Gamma_{ij}^3 = \frac12 (\p_i g_{j3}+\p_j g_{i3} - \p_3 g_{ij})$. \\
	Now, directly from \eqref{metricdef} we see that $\p_ig_{j3}=\p_jg_{i3}=0$ for $1\leq i,j\leq 2$.\\Therefore, in a general point $p=(x',0)\in S$, the second fundamental form is given by
	\[\II_{ij}(p)=\Gamma_{ij}^3(p) = -\frac12 \p_3\big{|}_{p} g_{ij} = -\frac12\cdot \left(-2 A_{ij}(x')\right) = A_{ij}(x'),\]
	so $\II(x')=A(x')$ as requested. 
\end{proof}

\subsubsection{The Quadratic Form $Q_3$}\label{ss:quadratic-forms}
Recall that the standard inner product on matrices in $\mathbb{R}^{n\times n}$ (also known as the \textit{Frobenius inner product}) is given by
\[\inn{A:B} = \Tr(A^TB) = \sum_{i,j}A_{ij}B_{ij},\]
where, throughout this section, we omit the range of summation when there is no room for confusion.
Given any linear map $\mathbb{M}:\mathbb{R}^{n\times n}\to \mathbb{R}^{n\times n}$, we may write its components with respect to the standard basis by $\mathbb{M}_{ijkl}$ (in other words, any such $\mathbb{M}$ is a fourth-order tensor). We say that a linear map $\mathbb{M}:\mathbb{R}^{n\times n}\to \mathbb{R}^{n\times n}$ is \textit{symmetric} if $\inn{\mathbb{M}A:B}=\inn{A:\mathbb{M}B}$ for all $A,B\in \mathbb{R}^{n\times n}$; this is equivalent to requiring that $\mathbb{M}_{ijkl}=\mathbb{M}_{klij}$.
\begin{definition}
	A function $Q: \mathbb{R}^{n\times n}\to \mathbb{R}$ is called a \textit{quadratic form} if there exists a symmetric linear map $\mathbb{M}:\mathbb{R}^{n\times n}\to \mathbb{R}^{n\times n}$ such that for all $A\in \mathbb{R}^{n\times n}$,
	\[Q(A) = \inn{A: \mathbb{M}A} = \sum_{i,j,k,l} \mathbb{M}_{ijkl}A_{ij}A_{kl}.\]
\end{definition}
\begin{remark} A quadratic form $Q$ has a unique symmetric linear map $\mathbb{M}$ defining it.\end{remark}
\begin{remark} Any quadratic form $Q$ uniquely determines a symmetric bilinear form $b$ on $\mathbb{R}^{n\times n}\times \mathbb{R}^{n\times n}$ by polarization:
	\[b(A,B) = \frac12\bra{Q(A+B)-Q(A)-Q(B)}.\]
	Conversely, any symmetric bilinear form $b$ determines a quadratic from $Q(A)=b(A,A)$. We say that such $b$ and $Q$ are \textit{associated} with each other. We will henceforth refer both to a quadratic form $Q$ and to its associated symmetric bilinear form $b$ simply as \textit{quadratic forms}, and by abuse of notation we will write $Q(A,B)$ for $b(A,B)$.
\end{remark}
\begin{definition}
	A quadratic form $Q$ on $\mathbb{R}^{n\times n}$ is \textit{positive-definite} if $Q(A) >0$ for all $A\neq 0$. We similarly define \textit{non-negative definite} quadratic forms.
\end{definition}
\begin{lemma} Let $W:\mathbb{R}^{n \times n}\to \mathbb{R}$ be any map of class $C^2$ in a neighbourhood of $I$.\\ Then $Q(F)\coloneqq \frac{\p^2 W}{\p F^2}(I)(F,F)$ is a quadratic form. 
\end{lemma}
\begin{proof} By definition of $Q$, we have
	\begin{align*}
	Q(F) = \sum_{i,j,k,l} \frac{\p^2 W}{\p F_{ij}\p F_{kl}}(I) F_{ij}F_{kl} = \inn{F:\mathbb{M}F},
	\end{align*}
	where $\mathbb{M}_{ijkl}\coloneqq\frac{\p^2 W_n}{\p F_{ij}\p F_{kl}}(I)$ is symmetric by Clairaut's theorem.
\end{proof}

Consider now our elastic energy density $W$.
Since $W$ attains a local minimum at each rotation in $\SO(3)$, and particularly at $I$, the first derivative of $W$ at $I$ vanishes identically. By Taylor expansion, this means that for matrices $F$ at the vicinity of $I$, the value of  $W(F)$ is dictated by the second derivative of $W$ at $I$, or equivalently by the quadratic form 
\begin{equation}\label{eq:Q3-def}
Q_3(F)\coloneqq\frac{\p^2 W}{\p F^2}(I)(F,F).
\end{equation} More precisely, the assumptions on $W$ imply that we can write $W(I+F)=\frac12 Q_3(F) + \eta(F)$, where $\eta(F)/|F|^2\to 0$ as $|F|\to 0$. Letting $\o(s)=\sup_{|F|\leq s}\abs{\eta(F)}$, we have
\begin{equation}\label{eq:W-expansion}
\frac12 Q_3(F) -\o(|F|)\leq W(I+F)\leq \frac12 Q_3(F) +\o(|F|),
\end{equation}
where $\o(s)/s^2\to 0$ as $s\to 0$. 
\begin{lemma}\label{lem:Q3-is-quadratic}
	The quadratic form $Q_3$ is  non-negative definite, and its value depends only on the symmetric part of its argument. On the space of symmetric matrices, $Q_3$ is positive-definite. \end{lemma}
\begin{proof}
	Let $0\neq F\in \mathbb{R}^{3\times 3}$. Let $t$ be sufficiently small, so that both $\det(I+tF)$ is positive and $(I+t \Sym F)$ lies in a neighbourhood of $I$ in which $W$ is of class $C^2$. A simple calculation shows that $\dist(A,\SO(3))=|\sqrt{A^TA}-I|$ for $A$ with positive determinant. Since $(I+tF)^T(I+tF)=I+2t\Sym F +t^2 F^TF$, we have $\sqrt{(I+tF)^T(I+tF)} = I+t \Sym F +O(t^2)$. Therefore,
	$\dist^2(I+t F,\SO(3))=\abs{t \Sym F}^2+o(t^2)$. By coercivity of $W$ we have 
	$C\dist^2(I+tF,\SO(3))\leq W(I+t F)$ for some positive constant $C$, from which we deduce that
	\[C \abs{\Sym F}^2 \leq \liminf_{t\to 0} \frac{1}{t^2} W(I+t F).\]
	By \eqref{eq:W-expansion}, it is immediately seen that $\lim_{t\to0}t^{-2}W(I+tF)=\frac12 Q_3(F)$, so we get 
	\begin{equation}\label{eq:Q3-coercivity}
	Q_3(F)\geq 2 C |\Sym F|^2.
	\end{equation}
	This proves the claims about the definiteness of $Q_3$.
	Now, by frame indifference of $W$ and polar decomposition, we can write 
	$W(I+t F) = W(\sqrt{(I+tF)^T(I+tF)}) = W(I+t \Sym F + O(t^2))$. Employing \eqref{eq:W-expansion} again, we have
	\begin{align*}
	\frac12 t^2 Q_3(\Sym F + O(t)) - \o\bra{|t\Sym F|+O(t^2)}&\leq W(I+t F) 
	\\&\leq \frac12 t^2 Q_3(\Sym F + O(t)) + \o\bra{|t\Sym F|+O(t^2)}.
	\end{align*}
	Dividing by $t^2$, letting $t\to0$, and recalling that $Q_3$ is continuous (since $W$ is $C^2$ near the identity), we get 
	\[\frac12 Q_3(\Sym F) = \lim_{t\to 0} W(I+t F) = \frac12 Q_3(F),\]
	which proves that $Q_3$ depends only on the symmetric part of its argument.
 An alternative proof for Lemma \ref{lem:Q3-is-quadratic} can be found in \cite[Lemma~5.2]{lewicka-prestressed}.
\end{proof}

The quadratic form $Q_3$ gives rise to several downstream quadratic forms. For instance,  we may define $q_2: \mathbb{R}^2\to \mathbb{R}$ by
\begin{equation}\label{eq:q2-def}
q_2(\alpha,\beta) = \min_{z\in \mathbb{R}^4} Q_3\bra{\begin{matrix}\alpha & \beta & z_1\\ \beta& z_2 & z_3\\z_1 & z_3 & z_4\end{matrix}}.
\end{equation}
\begin{lemma}\label{lem:q2-is-quadratic} $q_2$ is a positive-definite quadratic form on $\mathbb{R}$.
\end{lemma}
\begin{proof} Since $Q_3$ is a quadratic form, for each $\alpha\in \mathbb{R}$ and $z\in \mathbb{R}^4$ we can write
	\begin{equation}
	\label{eq:q2-Q3-expansion}
	Q_3(\alpha,z)\coloneqq Q_3\bra{\begin{matrix}\alpha & \alpha & z_1\\ \alpha& z_2 & z_3\\z_1 & z_3 & z_4\end{matrix}} = c\alpha^2+\sum_{i=1}^4 c_i \alpha z_i + \sum_{i,j=1}^4 c_{ij} z_i z_j
	\end{equation}
	for some constants $c,c_i,c_{ij}$. Notice that for  all $z\in \mathbb{R}^4$, by coercivity of $Q_3$ (equation \eqref{eq:Q3-coercivity}), we have
	\[Q_3(\alpha,z)\geq C (3|\alpha|^2+|z|^2+|z_1|^2+|z_3|^2)\geq C |z|^2\]
	for some positive constant $C$. Therefore, by the Bolzano-Weierstrass theorem and the continuity of $Q_3$, if $(z^{(k)})\subseteq \mathbb{R}^4$ satisfies $Q_3(\alpha,z^{(k)})\to \inf_{z} Q_3(\alpha ,z)$, then (up to subsequences) $z^{(k)}$ converges to some $\tilde{z}\in \mathbb{R}^4$ in which the infimum is attained. Since $Q_3$ is strictly convex on symmetric matrices (because it is positive-definite there), this minimum is in fact unique. We may therefore let $\tilde{z}(\alpha)=\{\tilde{z}_i
	(\alpha)\}_{i=1}^4$ be the (unique) solution to the minimum problem \eqref{eq:q2-def} for $q_2(\alpha) \coloneqq q_2(\alpha,\alpha)$. 
	If we manage to prove that $\tilde{z}(\alpha)$ depends linearly on $\alpha$, then \eqref{eq:q2-Q3-expansion} will imply at once that $q_2$ is indeed a quadratic form. Let $z\in \mathbb{R}^4$ be arbitrary. Since the function $t\mapsto Q_3(\alpha,\tilde{z}+t z)$ attains its minimum at $t=0$, we get
	\begin{align*}
	0 = \frac{d}{dt}{\bigg|}_{t=0} Q_3(\alpha,\tilde{z}+t z) &= \sum_i c_i \alpha z_i + \sum_{i,j} c_{ij}(\tilde{z}_i z_j + \tilde{z}_j z_i)
	\\&=\sum_i \bra{c_i \alpha+ \sum_j (c_{ij}+c_{ji})\tilde{z}_j}z_i
	\end{align*}
	Since this holds  for all $z\in \mathbb{R}^4$, this implies that 
	$0=c_i \alpha+ \sum_j (c_{ij}+c_{ji})\tilde{z}_j$, from which we deduce that $\tilde{z}$ is linear in $\alpha$. Conversely, every stationary point of the functional $z\mapsto Q_3(\alpha,z)$ is a minimizer, because the functional is strictly convex as explained above. Since $Q_3$ is positive-definite on symmetric matrices, this immediately implies that $q_2$ is positive-definite.
\end{proof}
Similarly, we may define 
$Q_2:\mathbb{R}^{2\times 2}\to \mathbb{R}$ by
\begin{equation}\label{eq:bigQ2-def}
Q_2(F) = \min_{z\in \mathbb{R}^3} Q_3 \begin{pNiceArray}{cc|c}[margin,columns-width=0.3cm]
\Block{2-2}{F}& &z_1 \\ & &z_2\\\hline 0&0&z_3
\end{pNiceArray}.
\end{equation}
\begin{lemma}\label{lem:bigQ2-is-quadratic}
	$Q_2$ is a quadratic form on $\mathbb{R}^{2\times 2}$ which vanishes on skew-symmetric matrices and is positive-definite on symmetric matrices.
\end{lemma}
\begin{proof}
	The proof is similar to the proof of Lemma \ref{lem:q2-is-quadratic}.  Since $Q_3$ is a quadratic form, for each $F\in \mathbb{R}^{2\times 2}$ and $z\in \mathbb{R}^3$ we can write
	\[Q_3(F,z)\coloneqq Q_3 \begin{pNiceArray}{cc|c}[margin,columns-width=0.3cm]
	\Block{2-2}{F}& &z_1 \\ & &z_2\\\hline 0&0&z_3
	\end{pNiceArray}= Q(F)+\sum c_{ij}z_i z_j + \sum c_{ijk} F_{ij} z_k,\]
	where $Q$ is some quadratic form on $\mathbb{R}^{2\times 2}$ and $c_{ij},c_{ijk}$ are some constants. For $F\in \mathbb{R}^{2\times 2}$, let $\tilde{z}(F)$ be the solution for the minimum problem \eqref{eq:bigQ2-def} for $Q_2(F)$ (where the proof of the existence of a unique minimizer is the same as in the proof of Lemma \ref{lem:q2-is-quadratic}). We only need to prove that $\tilde{z}$ depends linearly on $F$. Since $\tilde{z}$ is a minimizer, for any $z\in \mathbb{R}^3$ we have
	\begin{align*}
	0 =\frac{d}{dt}{\bigg|}_{t=0}  Q_3(F,\,\tilde{z}+t z) &= \sum c_{ij}(\tilde{z}_iz_j+z_i\tilde{z}_j) +\sum c_{ijk}F_{ij} z_k 
	\\&=\sum_i \bra{\sum_j(c_{ij}+c_{ji})\tilde{z}_j+\sum_{j,k}c_{k,j,i}F_{kj}}z_i
	\end{align*}
	Since this holds for all $z\in \mathbb{R}^3$, we conclude that
	$0=\sum_j(c_{ij}+c_{ji})\tilde{z}_j+\sum_{j,k}c_{k,j,i}F_{kj}$, proving that $\tilde{z}$ depends linearly on $F$. Conversely, every stationary point of the functional $z\mapsto Q_3(F,z)$ is a minimizer, because the functional is convex (indeed, $Q_3$ is non-negative definite, hence convex). Moreover, the definiteness of $Q_3$ immediately implies the required definiteness of $Q_2$.
\end{proof}

\newpage

\section{Settings and Problem Setup}\label{sec:setting}
Throughout this work, $W:\mathbb{R}^{3\times 3} \to [0,\infty]$ denotes an elastic energy density satisfying conditions (i)-(iv) as in \S~\ref{ss:hyperelastic}.
The induced quadratic forms $Q_3$ and $Q_2$ are defined as in 
\eqref{eq:Q3-def} and \eqref{eq:bigQ2-def}; see \S\ref{sec:background} 
for further background.

\subsection{Problem Setup} 
Consider a family of thin bodies $\O_{t,w} = (0,L)\times \left(-\frac{w}{2},\frac{w}{2}\right)\times \left(-\frac{t}{2},\frac{t}{2}\right)$, indexed over small enough $t,w$, where each $\O_{t,w}$ is equipped with a Riemannian metric $g_{t,w}$. 
We assume that the metric $g_{t,w}$ is given by 
\begin{equation}\label{eq:g_tw}
g_{t,w}(z_1,z_2,z_3) = I -2 t B(z_1,z_2/w,z_3/t)
\end{equation}
for some $B\in L^\infty\left(\left(0,L\right)\times \left(-\frac12,\frac12\right)\times\left(-\frac{1}{2},\frac{1}{2}\right); \mathbb{R}^{3\times 3}_{\sym+}\right)$.
We note that if $g_{t,w}$ is a restriction of a smooth metric $g$ on $\O_{t_0,w_0}$ to $\O_{t,w}$, for which the ribbon is a $t$-tubular
	neighbourhood of the mid-surface $\{z_3 = 0\}$, which, in turn, is a $w$-tubular neighbourhood of the
	midline $\ell = \{z_2 = z_3 = 0\}$, the metric would have had the form
\[
g(z)=\begin{pmatrix}(1-\kappa(z_1)z_2)^2-K(z_1,0)z_2^2+O(z_2^3)&0&0\\0&1&0\\0&0&1\end{pmatrix}-2 z_3 \begin{pmatrix}\II(z_1,z_2)&0\\0&0\end{pmatrix}+O(z_3^2),
\]
where $\kappa$ is the geodesic curvature of the midline within the midsurface, $K$ is the Gaussian curvature of the midsurface, and $\II$ is the second fundamental form of the midsurface in $\O_{t_0,w_0}$ 
(cf. \cite[\S~2]{mm25}). 
Therefore, heuristically, we can understand \eqref{eq:g_tw} as describing a ribbon whose midline is asymptotically straight, in the sense that the geodesic curvature is $\kappa_g=O(t/w)$, and that the $z_1,z_2,z_3$-directions are near-orthogonal (up to an error of order $O(t)$). 
The second fundamental form can still be of order $O(1)$.
For each $(t,w)$, let $P_{t,w}:T\O_{t,w} \to \mathbb{R}^3$ be a prestrain compatible with $g_{t,w}$.
We are interested in the limiting elastic energy (in the sense of $\Gamma$-convergence, and after an appropriate scaling) as $t,w\to 0$, where we assume that the width and thickness satisfy 
\[
t \ll w
\] 
along their way to zero.
It is convenient to perform the analysis on a fixed domain, so let us rescale the problem to $\O\coloneqq \O_{1,1}$ via the change of coordinates $(z_1,z_2,z_3)\mapsto (x_1,w x_2, tx_3)$. 
Given a configuration $y=y(x): \O \to \mathbb{R}^3$ on the rescaled domain, we define the \textit{scaled deformation gradient},
\[ \D_{t,w} y \coloneqq \left(y_{,1}\mid \frac{1}{w} y_{,2}\mid \frac{1}{t}y_{,3}\right).\]

 Let us be given now a family $v^{t,w}$ of configurations, 
  each respectively defined on $\O_{t,w}$. Denoting $y^{t,w}(x)=v^{t,w}(x_1,wx_2,tx_3)$, we obtain a family of configurations on $\O$, such that
$\D_{t,w} y^{t,w}(x)=\D v^{t,w}(z)  $. The metric $g_{t,w}$ is given locally with respect to the new coordinates $(x_1,x_2,x_3)$ by \[g_{t,w}(x) =I -2t B(x_1,x_2,x_3),\] and the redefined elastic energy of a configuration $y:\O \to \mathbb{R}^3$ is given by the integral
\[\int_\O W(\D_{t,w} y(x) P_{t,w}^{-1}(x))\cdot \left|\det g_{t,w}(x)\right|^{1/2}dx.\]

 In this work, we assume for simplicity that $P^{-1}_{t,w}(x) = g_{t,w}^{-1/2}(x)$ in coordinates.\footnote{The general behaviour
 	  of the problem is not affected by this choice. In fact, because any two plastic strains compatible with $g_{t,w}$ differ by a rotation-field, for isotropic $W$ (i.e., right-$\SO(3)$-invariant) all $g_{t,w}$-compatible plastic strains result in the same energy.}
 Expanding, we find
\[g_{t,w}^{-1/2}(x)=I+ t B(x) + O(t^2),\qquad \quad \abs{\det g_{t,w}(x)}^{1/2}=1- t\Tr(B(x))+O(t^2).\]
Because our focus is limited to the Kirchhoff bending regime $\eps_{t,w}\sim t^2$, we may omit higher order $t$-terms; we will henceforth work with the simplified elastic energy
\begin{equation}
I^{t,w}(y) \coloneqq \int_\O W\bra{\D_{t,w} y(x) \bra{I+ t B(x)}}dx.
\end{equation}

\subsection{Notation}
We write $\check{F}$ for the leading $2\times 2$ submatrix of a matrix $F\in \R^{3\times 3}$. We write $x=(x_1,x_2,x_3)$, and abbreviate $x'=(x_1,x_2)$. The midsurface of $\O$ is denoted by $S\coloneqq (0,L)\times (-\frac12,\frac12)$, so that $\O=S\times \bra{-\frac12,\frac12}$.
For any $y\in W^{1,2}(\O;\R^3)$, we denote by $y_{,i}$ or $\p_i y$ the column vector of the partial derivatives of $y$ with respect to $x_i$, $i=1,2,3$. 
For a matrix-field $f:\O\to \R^{n\times n}$, 
 we write $f_{ij}=e_i\cdot fe_j$ for the $(i,j)$-component function of $f$. 
Given a quadratic form $Q(F)$ on $\R^{n\times n}$, we write $Q(F,G)$ for the associated bilinear form, and $Q(F)=\sum_{i,j,k,l} \mathbb{M}_{ijkl}F_{ij}F_{kl}$ where $\mathbb{M}$ is the associated symmetric fourth-order tensor.

\newpage
\section{Technical Preliminaries}\label{sec:preliminaries}

\subsection{Preferred Bending and Excess Energy}\label{sec:excess-energy}
Consider the Hilbert space $\mathcal{H}_3= L^2(\O;\mathbb{R}^{2\times 2}_{\sym})$ with the inner product 
$\inn{f,g}=\int_\O Q_2(f,g)dx$
($Q_2$ is a positive-definite quadratic form on symmetric matrices, so this is indeed an inner product). Similarly, consider the Hilbert space $\mathcal{H}_2 = L^2(S;\mathbb{R}^{2\times 2}_{\sym})$  with the inner product $\inn{f,g}=\int_S Q_2(f,g)dx'$.
Let $\P_3:\mathcal{H}_3\to \mathcal{M}_3$ be the orthogonal projection onto the subspace 
\[
\mathcal{M}_3 = \set{
	x\mapsto 
	\begin{pmatrix}\gamma_1(x_1)+x_2\g_2(x_1)&a_{12}(x')\\a_{12}(x')&a_{22}(x')\end{pmatrix} +x_3 F(x') \;:\; 
	\g_i \in L^2((0,L)), a_{ij}\in L^2(S), F\in L^2(S;\R^{2\times 2}_\sym)
},
\]
and let $\P_2:\mathcal{H}_2\to \mathcal{M}_2$ be the orthogonal projection onto
\[
\mathcal{M}_2 = \set{
	x'\mapsto \begin{pmatrix}b_{11}(x_1)&b_{12}(x_1)\\b_{12}(x_1)&b_{22}(x')\end{pmatrix} \;:\;
	b_{11},b_{12}\in L^2((0,L)), b_{22}\in L^2(S)
}.
\]
These subspaces arise naturally from the structure \eqref{eq:structure-of-G} of the \textit{scaled limiting strain}, discussed in \S~\ref{sec:structure-strain}.
It is straight forward to check that for all $f\in L^2(S;\R^{2\times 2}_\sym)$ we have
\begin{equation}\label{eq:proj-2}
\P_2(f)(x') = \begin{pmatrix}
\int_{-\frac12}^{\frac12} f_{11}dx_2 
& \int_{-\frac12}^{\frac12} f_{12}dx_2
\\
\int_{-\frac12}^{\frac12} f_{12}dx_2  & f_{22}
\end{pmatrix},
\end{equation}
as well as
\begin{equation}\label{eq:proj3}
\P_3(f)(x) = \begin{pmatrix}
\int_{-\frac12}^{\frac12}\int_{-\frac12}^{\frac12} f_{11}dx_3dx_2 + x_2 \cdot 12\int_{-\frac12}^{\frac12}\int_{-\frac12}^{\frac12} x_2 f_{11}dx_3dx_2
& \int_{-\frac12}^{\frac12} f_{12}dx_3 
\\
\int_{-\frac12}^{\frac12} f_{12}dx_3  &\int_{-\frac12}^{\frac12} f_{22}dx_3 
\end{pmatrix}
+ x_3\cdot 12 \int_{-\frac12}^{\frac12} x_3 f dx_3
\end{equation}
for every $f\in L^2(\O;\R^{2\times 2}_\sym)$.

\begin{lemma}\label{lem:Q2-bar-def}
	For each $F\in L^2(S;\mathbb{R}^{2\times 2}_{\text{sym}})$, let
	\[\overline{Q}_2(F)\coloneqq \min_{A} \int_\O Q_2\bra{A(x')+x_3 F(x')+ \check{B}(x)} dx,\]
	where the minimum is taken over all $A\in L^2(S;\R^{2\times 2}_\sym)$ of the form
	\[A(x') = \begin{pmatrix}\gamma_1(x_1)+x_2\g_2(x_1)&a_{12}(x')\\a_{12}(x')&a_{22}(x')\end{pmatrix}\]
	for some $\g_i \in L^2((0,L)), a_{ij}\in L^2(S)$.
	Then 
	\[\overline{Q}_2(F) = \frac{1}{12} \int_S Q_2(F-\overline{\II})dx' + \alpha\]
	where $\overline{\II}\in L^2(S;\mathbb{R}^{2\times 2}_{\text{sym}})$ and $\alpha\geq 0$ are given by
	\begin{align}
	\alpha &\coloneqq \int_\O Q_2\bra{(\id-\P_3)(\check{B})}dx,\label{eq:alpha} 
	\\\overline{\II}(x') &\coloneqq  -12\int_{-\frac12}^{\frac12} x_3  \check{B}(x',x_3) dx_3.\label{eq:oII}
	\end{align}
\end{lemma}

\begin{proof} 
	Write $\P_3(\check{B})= \widetilde{A}+x_3 \widetilde{F}$ for some $\widetilde{A},\; \widetilde{F}\in L^2(S;\mathbb{R}^{2\times 2}_{\text{sym}})$, which we are able to specify via equation \eqref{eq:proj3}. Notice that for each $A\in L^2(S;\mathbb{R}^{2\times 2}_{\text{sym}})$ we have
	\begin{align*}
	\int_\O Q_2\bra{A+x_3 F + \check{B} }dx
	&=\int_\O Q_2\bra{A+\widetilde{A}+x_3(F+\widetilde{F})+(\id-\P_3)( \check{B})}dx
	\\&= \int_S Q_2(A+\widetilde{A})dx' +\frac1{12} \int_S Q_2(F+\widetilde{F})dx' +\int_\O Q_2\bra{(\text{Id}-\P_3)(\check{B})}dx
	\end{align*}
	Since $Q_2$ is positive-definite on symmetric matrices, the above quantity is minimized when $A=-\widetilde{A}$. Hence
	$\overline{Q}_2(F) = \frac1{12}\int_S Q_2(F-\overline{\II})dx' + \alpha$,
	where $\overline{\II}\coloneqq -\widetilde{F}$ and $\alpha$ are given by equations \eqref{eq:oII} and \eqref{eq:alpha}.
\end{proof}

The knowledge of the energy density $W$ and the prestrain $B$ induces the notions of \textit{reference} second fundamental form, \textit{reference} geodesic curvature, and \textit{excess energy}:
\begin{definition}\label{def:effective-excess}\leavevmode
	\begin{enumerate}
\item The \emph{reference second fundamental form} (or the \emph{preferred bending}) \emph{along the midline} is the function
$\wtfII\in L^2((0,L); \R^{2\times 2}_\sym)$ given by
\begin{equation}\label{eq:effective-form}
\wtfII(x_1)=\int_{-\frac12}^{\frac12}\oII(x_1,x_2)dx_2
,
\end{equation}
where $\oII(x') = -12 \int_{-\frac12}^{\frac12} x_3 \check{B}(x',x_3)dx_3$. 
\item The \emph{excess} (or \emph{residual}) \emph{energy} is the non-negative  constant given by
\begin{equation}\label{eq:excess-energy}
C_{\text{excess}} \coloneqq
\frac{1}{24} \int_S Q_2\bra{(\id-\P_2)(\oII) }  dx' + 
\frac12 \int_\O Q_2\bra{ (\id-\P_3)(\check{B}) }  dx.
\end{equation}
\item The \textit{reference geodesic curvature at scale $t/w$} is the function $\kappa_g\in L^\infty((0,L);\R)$ given by
\[\kappa_g(x_1) = 12\int_{-\frac12}^{\frac12}\int_{-\frac12}^{\frac12} x_2 B_{11}(x_1,x_2,x_3)dx_3dx_2.\]
\end{enumerate}
\end{definition}

This naming is motivated by the fact that, via \eqref{eq:proj3} and \eqref{eq:oII}, we have
\begin{equation}\label{eq:P3B-kg}
\P_3(\check{B})= \begin{pmatrix}\kappa_0+x_2\kappa_g&\tilde{a}_{12}\\\tilde{a}_{12}&\tilde{a}_{22}\end{pmatrix} -x_3 \oII
\end{equation}
for some $\kappa_0 \in L^\infty((0,L)), \tilde{a}_{ij}\in L^\infty(S)$, and
for a metric of the form $I-2t \P_3(\check{B})$ the geodesic curvature of the midline is $\frac{t}{w}\kappa_g+O(t)$ and the second fundamental form of the midsurface is is $\oII+O(t)$ (cf. \S\ref{sec:setting} and also the discussion in \cite[\S2]{mm25}).
We remark that, by definition of the orthogonal projections $\P_2$ and $\P_3$, 
the non-negative constant $C_{\text{excess}}$  is positive unless the $2\times 2$ leading submatrix of $B$ has the form
\begin{equation}\label{eq:non-generic-prestrain}
\check{B}(x_1,x_2,x_3)=
\begin{pmatrix}
\g_1(x_1)+x_2\g_2(x_1) & a_{12}(x_1,x_2)\\
a_{12}(x_1,x_2) & a_{22}(x_1,x_2)
\end{pmatrix}
+x_3 \,
\begin{pmatrix}
b_{11}(x_1) & b_{12}(x_1)\\
b_{12}(x_1) & b_{22}(x_1,x_2)
\end{pmatrix}
\end{equation}
for some $\g_i,a_{ij},b_{ij}\in L^\infty$.

\subsection{Narrow Ribbons: Compactness Lemmas and the Limiting Strain}\label{sec:narrow-prelim}
\paragraph{The Associated Second Fundamental Form.}
We say that a sequence $(y^{t,w})\subseteq W^{1,2}(\O;\mathbb{R}^3)$ has \textit{finite energy} 
 if
\begin{equation}\label{finitebending}
\limsup_{t,w\to 0} \frac{1}{t^2} I^{t,w}(y^{t,w}) < \infty,
\end{equation}
where, in this section, "$t,w\to0$" always means in the narrow sense $w^2\ll t \ll w$. 
Below, we state compactness results for finite energy sequences which directly follow from the Euclidean case \cite{fmp12} using the fundamental geometric rigidity estimate \cite{fjm2002}: 
\begin{lemma}\label{lem:rotations-existence}
For any sequence $(y^{t,w})\subseteq W^{1,2}(\O;\mathbb{R}^3)$ with finite energy, there exists an associated sequence $(R^{t,w})\subseteq C^\infty(S;\SO(3))$ of rotation-fields such that
\begin{enumerate}[label=(\roman*)]
	\item $\displaystyle\norm{\D_{t,w}y^{t,w}-R^{t,w}}_{L^2} \leq C t$,
	\item $\displaystyle\norm{R^{t,w}_{,1}}_{L^2} \leq C$, \quad
	$\displaystyle\norm{R^{t,w}_{,2}}_{L^2} \leq C w$.
\end{enumerate}
\end{lemma}
\begin{proof}
By the coercivity of $W$ and the $L^\infty$-boundedness of $B$, any sequence $y^{t,w}$ with finite energy satisfies $\int_\O \dist^2\bra{\D_{t,w}y^{t,w},\SO(3)} dx \leq C t^2$. The claim then follows by \cite[Lemma~3.3]{fmp12}. 
\end{proof}

In \cite{fmp12}, the authors show that, in the case of Euclidean narrow ribbons, the relevant $\Gamma$-limit is finite only if the limit configuration $y$ of the midline and the limit frame $(y',d_2,d_3)$ along the midline (where the notion of the limit is detailed below) are in the class
\begin{align}
\begin{split}
\mathcal{A} \coloneqq \big\{\,(y,d_2,d_3)\in  W^{2,2}((0,L);\mathbb{R}^3) \times W^{1,2}((0,L);\mathbb{R}^3)\times W^{1,2}((0,L);\mathbb{R}^3)\;:\\\; (y_{,1}|d_2|d_3)\in \SO(3)\,\;\text{and}\;\, y_{,1}\cdot d_{2,1} =0\;\,\text{a.e. in}\;\, (0,L)\,\big\}.
\end{split}
\end{align}
The corresponding compactness result can be summarized as follows:
\begin{lemma}\label{lem:associated-form}
(Cf. \cite[Proof of Theorem 5.1(i)]{fmp12})	
	 Let $(y^{t,w})\subseteq W^{1,2}(\O;\mathbb{R}^3)$ be a sequence with finite energy. Then there exists a triple $(y,d_2,d_3)\in \mathcal{A}$ such that (up to subsequences) $\D_{t,w}y^{t,w}\to R\coloneqq (y_{,1}|d_2|d_3)$, 
and there exists a function $K\in L^2(S;\mathbb{R}^{3\times 3})$ for which (up to subsequences) 
\begin{equation}
\frac{1}{w}(\D_{t,w}y^{t,w})_{,2}\wto K \quad \text{weakly in}\; W^{-1,2}(\O, \mathbb{R}^{3\times 3}).
\end{equation}
Moreover, $K$ admits the structure 
\begin{equation*}
K = \bra{d_{2,1}\; {\big|}\; \gamma d_3 \;{\big|}\; -(d_3\cdot d_{2,1})y_{,1}-\gamma d_2},
\end{equation*}
where $\gamma$ is some function in $L^2(S)$.
\end{lemma}

\begin{proof}[Proof of Lemma \ref{lem:associated-form}.]
We follow the arguments as  in \cite[Proof of Theorem 5.1(i)]{fmp12}. Since $(y^{t,w})$ has finite energy, we may apply Lemma \ref{lem:rotations-existence} and obtain a sequence of smooth maps $R^{t,w}:S\to \mathbb{R}^{3\times 3}$ such that
\begin{align}
&R^{t,w}(x') \in \SO(3) \quad \text{for all}\quad x'\in S, \label{eq:R-in-SO}
\\&\norm{\D_{t,w}y^{t,w}-R^{t,w}}_{L^2} \leq C t,\label{eq:R-sim-D}
\\ &\norm{R^{t,w}_{,1}}_{L^2} \leq C, \quad
\norm{R^{t,w}_{,2}}_{L^2} \leq C w. \label{eq:R-derivatives}
\end{align}
From \eqref{eq:R-in-SO} and \eqref{eq:R-derivatives} it follows that there exists $R\in W^{1,2}((0,\ell); \mathbb{R}^{3\times 3})$ such that $R(x_1)\in \SO(3)$ for every $x_1\in (0,\ell)$ and, up to subsequences, $R^{t,w}\wto R$ weakly in $W^{1,2}(S;\mathbb{R}^{3\times 3})$ (and thus strongly in $L^2(S;\mathbb{R}^{3\times 3})$). By \eqref{eq:R-sim-D} we deduce that 
\begin{equation}
\D_{t,w}y^{t,w}\to R\quad\text{in}\quad L^2(S;\mathbb{R}^{3\times 3}).
\end{equation}
This implies that $\D y^{t,w}\to \bra{Re_1\;\big|\;0\;\big|\;0}$ strongly in $L^2(S;\mathbb{R}^{3\times 3})$. 
Therefore there exists $y\in W^{2,2}(\O;\mathbb{R}^3)$ 
 such that, up to additive constants, $y^{t,w}\to y$ strongly in $W^{1,2}(\O;\mathbb{R}^3)$ and $y_{,1}=Re_1$ and $y_{,2}=y_{,3}=0$ a.e. in $\O$. Denote $d_2=Re_2$, $d_3=Re_3$, so that we have $(y_{,1}|d_2|d_3)=R\in \SO(3)$. To prove that $(y,d_2,d_3)$ is in $\mathcal{A}$, we need to prove that $y_{,1}\cdot d_{2,1}=0$ a.e. in $\O$. By the second inequality in \eqref{eq:R-derivatives}, there exists $K\in L^2(S;\mathbb{R}^{3\times 3})$ such that, up to subsequences,
$\frac{1}{w}R^{t,w}_{,2} \wto K\;$ weakly in $\; L^2(S;\mathbb{R}^{3\times 3})$.
Differentiating the identity $\bra{R^{t,w}}^TR^{t,w}=I$, dividing by $w$, and letting $t,w\to 0$, we deduce that 
\begin{equation}\label{eq:RK-identity}
R^TK+K^TR=0.
\end{equation}
Continuing in exactly the same manner as in \cite{fmp12}, we find that
\begin{equation}\label{eq:K-form}
K = \bra{d_{2,1}\; {\big|}\; \gamma d_3 \;{\big|}\; -(d_3\cdot d_{2,1})y_{,1}-\gamma d_2},
\end{equation}
where $\gamma\in L^2(S)$. Hence, equating the first columns in equation \eqref{eq:RK-identity}, we obtain $y_{,1}\cdot d_{2,1} =0$ as wished.	
\end{proof}

Let now $(y^{t,w})\subseteq W^{1,2}(\O;\mathbb{R}^3)$ be a sequence with finite energy, and let $(y,d_2,d_3)\in \A$ and $K$ be as in Lemma~\ref{lem:associated-form}. As we will see, the structure of the scaled limiting strain leads us to consider the \emph{second fundamental form associated with $y^{t,w}$},
\begin{equation}\label{eq:associated-form}
\II \coloneqq \begin{pmatrix}y_{,1}\cdot d_{3,1} & d_2\cdot d_{3,1}\\d_{2}\cdot d_{3,1}& d_2\cdot Ke_3\end{pmatrix}.
\end{equation}
Note that $\II \in L^2(S;\R^{2\times 2}_\sym)$, and that only the component $\II_{22}=d_2\cdot Ke_3$ is (possibly) dependent on $x_2$.

\paragraph{Structure of the Scaled Limiting Strain.}\label{sec:structure-strain}

Let $(y^{t,w})\subseteq W^{1,2}(\O;\R^3)$ be a sequence with finite energy, and let $R^{t,w}\in C^\infty(S; \SO(3))$ be as in Lemma~\ref{lem:rotations-existence}. Define the sequence of \textit{scaled strains} 
\begin{equation}\label{eq:G-def}
G^{t,w} \coloneqq \frac{\bra{R^{t,w}}^T\D_{t,w}y^{t,w}-I}{t}.
\end{equation}
Condition (i) in Lemma~\ref{lem:rotations-existence} implies that $G^{t,w}$ is uniformly bounded, so (up to subsequences) $G^{t,w}$ converges weakly to some $G$ in $L^2(\O;\mathbb{R}^{3\times 3})$. 
We call $G$ the \textit{scaled limiting strain}. In \cite[Proof of Theorem 5.1(i)]{fmp12} it is shown that $\check{G}$ admits the structure
\begin{equation}\label{eq:G-structure}
\check{G}(x',x_3) = x_3 \II + \check{G}(x',0),
\end{equation}
where $\II$ is the associated second fundamental form of $y^{t,w}$, and $\check{G}$ is obtained from $G$ by omitting the third row and third column. 
Our limiting model depends on the fact that the scaled limiting strain admits additional structure; namely, the $(1,1)$-component of $\check{G}(x',0)$ is an $x_2$-affine function. 
Because $\II_{11}$ is constant in $x_2$, this is equivalent to showing that $\p_2^2 G_{11}=0$ in the sense of distributions;
this result was proven in \cite{mm25} in a more general setting (our case corresponds to $Q_0=I$, $\D_t\Psi_t =I$, and $\eps=t$).
Let us reproduce the proof here for our simpler case.
 To that end, we need an $x_2$-independent sequence of approximating rotations:
\begin{lemma}\label{lem:mm-R0} Let  $y^{t,w}$ be a sequence with finite  energy, and let
$R^{t,w}$ be as in Lemma \ref{lem:rotations-existence}. Then for each narrow pair $(t,w)$ there exists $x_2^{t,w}\in (-\frac12,\frac12)$ such that the function
\begin{align*}
R_0^{t,w}:(0,\ell)\longrightarrow& \;\SO(3)\\
x_1\longmapsto& \;R^{t,w}(x_1,\,x_2^{t,w})
\end{align*}
satisfies 
\[
1.\;\norm{ R_{0,1}^{t,w}}_{L^2} \leq C,
\qquad 2.\;\norm{R^{t,w}-R_0^{t,w}}_{L^2} \leq Cw,
\qquad 3.\;\norm{\Sym\bra{\bra{R^{t,w}_0}^TR^{t,w}-I}}_{L^1}\leq C w^2.
\]
\end{lemma}
\begin{proof}
This is precisely \cite[Lemma~4.6]{mm25}, adapted to our energy scaling $\eps \sim  t$.
\end{proof}

\begin{lemma}\label{lem:G} 
(Cf. \cite[Lemma~4.8]{mm25})
	Let $y^{t,w}$ be a sequence with finite energy and consider its scaled limiting strain $G$. The leading $2\times 2$ submatrix of $G$ admits the following structure:
\begin{equation}\label{eq:structure-of-G}
\check{G}(x_1,x_2,x_3) = x_3 \II(x_1,x_2) + \begin{pmatrix}\g_1(x_1)+x_2 \g_2(x_1) & a_{12}(x_1,x_2)\\ a_{21}(x_1,x_2) & a_{22}(x_1,x_2) \end{pmatrix},
\end{equation}
where $\g_i\in L^2((0,L))$, $a_{ij} \in L^2(S)$, and $\II$ is the second fundamental form associated with $y^{t,w}$ as defined in \eqref{eq:associated-form}.
\end{lemma}
We may give a geometric interpretation for the $\g_i$ terms:
The term $\g_1$ is the leading order of the $(1,1)$-component of the metric, which represents the order-$t$ stretching of the midline in the direction $x_1$. Meanwhile, $\gamma_2$ represents the first correction for the geodesic curvature of the midline, which is of order $t/w$.

\begin{proof} Let $R^{t,w}$ be the rotations emerging in Lemma \ref{lem:rotations-existence}, where for easy reference we recall that
	\begin{equation}\label{eq:Rtw-easy-ref}
	1.\;\norm{\D_{t,w}y^{t,w}-R^{t,w}}_{L^2} \leq C t,
	\qquad\;
	2.\;\norm{R^{t,w}_{,1}}_{L^2} \leq C,
	\qquad\;
	3.\;\norm{R^{t,w}_{,2}}_{L^2} \leq C w.
	\end{equation} 
We already know that
 $\check{G}(x',x_3)=x_3\II+\check{G}(x',0)$. Because $\II_{11}$ is constant in $x_2$, it follows that we only need to show that, as a distribution,  	$\p^2_2G_{11}=0$.  Use Lemma \ref{lem:mm-R0}(iii) to write
\begin{align*}
\bra{\bra{R^{t,w}_0}^T R^{t,w}G^{t,w}}_{11} &= e_1 \cdot 
\bra{\frac{\bra{R_0^{t,w}}^T \D_{t,w}y^{t,w}-I}{t}-\frac{\Sym\bra{\bra{R^{t,w}_0}^T R^{t,w}-I}}{t}}e_1
\\&= e_1 \cdot \frac{1}{t}\bra{\bra{R_0^{t,w}}^T \D_{t,w}y^{t,w}-I}e_1 + O_{L^1}\bra{\frac{w^2}{t}}.
\end{align*}
Noting that $\p_2 \D_{t,w}y^{t,w}e_1=w \p_1 \D_{t,w}y^{t,w}e_2$, and using the first inequality in \eqref{eq:Rtw-easy-ref}, we get
\begin{align*}
\p_2 \bra{\bra{R^{t,w}_0}^T R^{t,w}G^{t,w}}_{11} &=
e_1 \cdot \frac{w}{t} \bra{\bra{R_0^{t,w}}^T \p_1 \D_{t,w}y^{t,w}e_2}
+O_{W^{-1,1}}\bra{\frac{w^2}{t}}
\\&=
e_1 \cdot \frac{w}{t} \sqbra{\bra{R_0^{t,w}}^T \p_1\bra{ \D_{t,w}y^{t,w}-R^{t,w}}e_2  +\bra{R_0^{t,w}}^T \p_1 R^{t,w}e_2 } +O_{W^{-1,1}}\bra{\frac{w^2}{t}}
\\&= e_1\cdot \frac{w}{t}\bra{R_0^{t,w}}^T \p_1 R^{t,w}e_2  +O_{W^{-1,2}}(w)+O_{W^{-1,1}}\bra{\frac{w^2}{t}}.
\end{align*}
Taking $\p_2$ once again, and using the third inequality in \eqref{eq:Rtw-easy-ref}, we obtain
\begin{equation}\label{eq:p22G11-0}
\begin{split}
\p_2^2 \bra{\bra{R^{t,w}_0}^T R^{t,w}G^{t,w}}_{11} &=
e_1 \cdot \frac{w^2}{t}\bra{R_0^{t,w}}^T \p_1 \bra{\frac{1}{w}R^{t,w}_{,2}}e_2 +O_{W^{-2,2}}(w)+O_{W^{-2,1}}\bra{\frac{w^2}{t}}
\\&= O_{L^2}\bra{\frac{w^2}{t}}+O_{W^{-2,2}}(w)+O_{W^{-2,1}}\bra{\frac{w^2}{t}}
\\&\longrightarrow 0 \qquad\text{in}\; W^{-2,1}.
\end{split}
\end{equation}
Since $G^{t,w}\wto G$ weakly in $L^2$ and $\bra{R_0^{t,w}}^TR^{t,w}\to I$ strongly in $L^2$, we obtain
$\bra{R^{t,w}_0}^T R^{t,w}G^{t,w}e_1 \wto Ge_1$ weakly in $L^2$ (and therefore weakly in $L^1$). Hence 
\[\p_2^2 \bra{\bra{R^{t,w}_0}^T R^{t,w}G^{t,w}}_{11} \wto \p^2_2 G_{11}\qquad \text{in}\; W^{-2,1},\]
which, when comparing with \eqref{eq:p22G11-0}, gives the desired result.
\end{proof}

For the lower bound we will need the following standard result
(cf. \cite[Lemma~3.4]{fmp12}, \cite[Proof of Theorem 3.2(i)]{sch07}, 
\cite[Lemma~5]{pg22}):
\begin{lemma}\label{lem:G-frac12B} Assume that $(y^{t,w})\subseteq W^{1,2}(\O;\mathbb{R}^3)$, $(R^{t,w})\subseteq \SO(3)$ are two sequences such that
\[G^{t,w}\coloneqq \frac{R^{t,w}\D_{t,w}y^{t,w}- I}{t} \rightharpoonup G \quad\text{weakly in}\; L^2(\O;\mathbb{R}^{3\times 3}) .\]
Then
\begin{equation}
\liminf_{t,w\to 0} \frac{1}{t^2} I^{t,w}(y^{t,w}) \geq \frac12 \int_\O Q_3\bra{G(x)+ B(x)}dx.
\label{eq:lowerbound1}
\end{equation}
\end{lemma}

\begin{proof}
Let $\chi_{t,w}$ denote the indicator function of the set $\left\{x\in \O\;:\; |G^{t,w}(x)|\leq t^{-1/2}\right\}$.  From \eqref{eq:W-expansion} and frame-indifference of $W$ we have 
\begin{align}
\frac{1}{t^2}\int_\O W\bra{\D_{t,w}y^{t,w}\bra{I+t B}}dx
&\geq   \frac{1}{t^2}\int_\O \chi_{t,w}W\bra{\bra{R^{t,w}}^T\D_{t,w}y^{t,w}\bra{I+ t B}}dx
\nonumber  \\&=  \frac{1}{t^2}\int_\O \chi_{t,w}W\bra{I+ t A^{t,w}}dx
\nonumber  \\&\geq \int_\O\frac12 Q_3\bra{\chi_{t,w}A^{t,w}} - \chi_{t,w}\frac{1}{t^2}\o\bra{\abs{t A^{t,w}}} \label{eq:two-integrals-lower-bound}
\end{align}
where $A^{t,w}\coloneqq G^{t,w}+ \left(R^{t,w}\right)^T \D_{t,w}y^{t,w} B$. Because $G^{t,w}\wto G$ by assumption, it follows that $G^{t,w}$ is uniformly bounded, and in particular $R^{t,w}\D_{t,w}y^{t,w}\to I$ strongly:
\[\norm{R^{t,w}\D_{t,w}y^{t,w}- I}_{L^2} = t\cdot \norm{G^{t,w}}_{L^2} \to 0.\]
Therefore $A^{t,w}\wto G+ B$. Now, since $G^{t,w}$ is uniformly bounded, it also follows that $\chi_{t,w}$ converges in measure to the constant function $1$, and therefore 
\begin{equation}\label{eq:weak-convergence-A}
\chi_{t,w}A^{t,w}\wto G+ B.
\end{equation}
Since $Q_3$ is non-negative definite, hence convex, it is lower semi-continuous with respect to the convergence \eqref{eq:weak-convergence-A}. The second integral in \eqref{eq:two-integrals-lower-bound} goes to zero, since it is the product of $\abs{A^{t,w}}^2$, which is bounded in $L^1(\O)$, with $\chi_{t,w} \o\bra{\abs{t A^{t,w}}}/\abs{t A^{t,w}}^2$, which converges 
uniformly to zero
  since 
$\abs{t A^{t,w}}\leq C t^{1/2}$ whenever $\chi_{t,w}\neq 0$. The claim then follows by taking $\liminf$ as $t,w\to0$ in the narrow sense.
\end{proof}

\newpage
\section{Narrow Ribbons: Gamma Convergence}\label{sec:narrow-ribbons}
We are now ready to state the main theorem. As we will see, the class of limiting configurations will be given by the following set of quadruplets:
\begin{align*}
\wtA \coloneqq \big\{\,&(y,d_2,d_3,\widetilde{K}_3)\in  W^{2,2} \times W^{1,2}\times W^{1,2}\times L^2\;:\\
&(y_{,1}|d_2|d_3)\in \SO(3),\; y_{,1}\cdot d_{2,1}  =d_3\cdot \widetilde{K}_3 =0\;\text{and}\; y_{,1}\cdot \widetilde{K}_{3} = d_2 \cdot d_{3,1}
\;\,\text{a.e. in}\;\, (0,L)\,\big\},
\end{align*}
where each $W^{i,j}$ means $W^{i,j}((0,L);\mathbb{R}^3)$ and $L^2$ means $L^2((0,L);\mathbb{R}^3)$. 
Different from the analysis in \cite{fmp12}, in which the limiting energy depends only on the first column of the second fundamental form along the midline (which is induced by the frame $(y_{,1}|d_2|d_3)$, see \eqref{eq:associated-form}), here we aimed for a more detailed theory involving all components of the second form. This motivates the inclusion of the fourth component $\K_3$ in the configuration: each element $(y,d_2,d_3,\widetilde{K}_3)\in \wtA$ induces an \emph{associated averaged second fundamental form},
\begin{equation}\label{eq:fII}
\fII\coloneqq\begin{pmatrix}y_{,1}\cdot d_{3,1} & d_2\cdot d_{3,1}\\d_{2}\cdot d_{3,1}&  d_2\cdot \widetilde{K}_3\phantom{\bigg|}\end{pmatrix}.
\end{equation}

\begin{theorem}\label{thm:limit-model} Assume we are in the narrow ribbon case ($w^2\ll t \ll w$).
\begin{enumerate}
\item \textbf{(Compactness and lower bound)} Let $(y^{t,w})\subseteq W^{1,2}(\O;\mathbb{R}^3)$ be a sequence with finite energy. Then there exists a quadruple $(y,d_2,d_3,\widetilde{K}_3)\in \widetilde{\mathcal{A}}$ and a subsequence (not relabelled) such that $\D_{t,w}y^{t,w} \to (y_{,1}|d_2|d_3)$ in $L^2(\O;\mathbb{R}^{3\times 3})$, and $\widetilde{K}_3$ is obtained from $(y^{t,w})$ via
\begin{align}\label{eq:gen-y23-tw-convergence}
\begin{split}
\widetilde{K}_3(x_1)=\int_{-\frac12}^{\frac12} K_3(x_1,x_2)dx_2, \quad&\text{where}\; K_3\in L^2(S;\mathbb{R}^3)\;\text{satisfies}
\\\frac{1}{tw} y^{t,w}_{,23}\wto K_3\qquad &\text{weakly in}\; W^{-1,2}(\O,\mathbb{R}^3).
\end{split}
\end{align}
Moreover,
\begin{equation}\label{eq:gen-lowerbound}
\liminf_{t,w\to 0} \frac{1}{t^2} I^{t,w}(y^{t,w}) \geq \frac1{24}
\int_0^L Q_2(\fII-\wtfII)dx_1 + C_{\text{excess}}\eqqcolon I^0(y,d_2,d_3,\widetilde{K}_3)
\end{equation}
where $Q_2$ is the quadratic form defined in \eqref{eq:bigQ2-def},  $\wtfII$, $C_{\text{excess}}$ are the reference second fundamental form and excess energy defined in Definition~\ref{def:effective-excess}, and $\fII$ is the averaged second fundamental form defined in \eqref{eq:fII}.

\item \textbf{(Upper bound)}
For every $(y,d_2,d_3,\widetilde{K}_3)\in \widetilde{\mathcal{A}}$ there exists a sequence $(y^{t,w})\subseteq W^{1,2}(\O; \mathbb{R}^3)$ such that $y^{t,w}\to y$ in $W^{1,2}(\O;\mathbb{R}^3)$, $\D_{t,w}y^{t,w}\to (y_{,1}|d_2|d_3)$ in $L^2(\O;\mathbb{R}^{3\times3})$, \eqref{eq:gen-y23-tw-convergence} holds, and 
\[\lim_{t,w\to 0} \frac{1}{t^2} I^{t,w}(y^{t,w}) = I^0(y,d_2,d_3,\widetilde{K}_3),\]
 where the limit is understood in the narrow sense.
\end{enumerate}
\end{theorem}

\begin{proof}[Proof of Theorem~\ref{thm:limit-model} (\textbf{Compactness and lower bound}).]
Let $(y^{t,w})\subseteq W^{1,2}(\O;\mathbb{R}^3)$ be a sequence with finite  energy. Applying Lemma~\ref{lem:associated-form}, we obtain a triple $(y,d_2,d_3)\in \mathcal{A}$ such that (up to subsequences) $\D_{t,w}y^{t,w}\to R\coloneqq (y_{,1}|d_2|d_3)$, and a function $K\in L^2(S;\mathbb{R}^{3\times 3})$ for which (up to subsequences)
\begin{equation}\label{eq:y-weakly-to-K}
\frac{1}{w}(\D_{t,w}y^{t,w})_{,2}\wto K \quad \text{weakly in}\; W^{-1,2}(\O, \mathbb{R}^{3\times 3}),
\end{equation}
with the structure
\begin{equation}\label{eq:K-structure}
K = \bra{d_{2,1}\; {\big|}\; \gamma d_3 \;{\big|}\; -(d_3\cdot d_{2,1})y_{,1}-\gamma d_2}
\end{equation}
for some $\gamma\in L^2(S)$. Define $\K_3\in L^2((0,L);\R^3)$ by
\[\widetilde{K}_3 \coloneqq \int_{-\frac12}^{\frac12} Ke_3(x_1,x_2)dx_2.\]
The structure \eqref{eq:K-structure} of $K$ implies that $d_3\cdot Ke_3=0$ and $y_{,1}\cdot Ke_3 = d_2 \cdot d_{3,1}$, so integrating over $x_2\in (-\frac12,\frac12)$ we obtain $d_3\cdot \widetilde{K}_3=0$ and $y_{,1}\cdot \widetilde{K}_3=d_2\cdot d_{3,1}$. Since $(y,d_2,d_3)\in \A$, this implies that indeed $(y,d_2,d_3,\widetilde{K}_3)\in \wtA$. Furthermore, \eqref{eq:gen-y23-tw-convergence} holds with $K_3\coloneqq Ke_3$ because extracting the third columns from equation \eqref{eq:y-weakly-to-K} gives
 \begin{equation}
 \frac{1}{wt} y^{t,w}_{,23} \wto Ke_3 \quad\text{weakly in}\; W^{-1,2}(\O;\mathbb{R}^3).
 \end{equation}
It remains to prove the lower bound \eqref{eq:gen-lowerbound}.
Recall that the scaled strains $G^{t,w}$ converge weakly in $L^2(\O;\R^{3\times 3})$ to the scaled limiting strain $G$, which admits the structure $\check{G}(x',x_3)=x_3 \II+\check{G}(x',0)$, where $\II=\begin{psmallmatrix}y_{,1}\cdot d_{3,1} & d_2\cdot d_{3,1}\\d_{2}\cdot d_{3,1}& d_2\cdot Ke_3\end{psmallmatrix}$ is the second fundamental form associated with $y^{t,w}$. Applying Lemma~\ref{lem:G-frac12B} we obtain
\begin{align}\label{eq:lower-first}
\begin{split}
\liminf_{t,w\to 0} \frac{1}{t^2} I^{t,w}(y^{t,w}) &\geq \frac12 \int_\O Q_3\bra{G(x)+ B(x)}dx
\\&\geq \frac12 \int_\O Q_2\bra{\check{G}(x)+ \check{B}(x)}dx
\\&=\frac12 \int_\O Q_2\bra{\check{G}(x',0)+x_3 \II+\check{B}(x)}dx.
\end{split}
\end{align}
Recall from Lemma~\ref{lem:G} that $\check{G}(x',0)\in L^2(S;\R^{2\times 2}_\sym)$ is of the form
$\check{G}(x',0)=\begin{psmallmatrix}\g_1(x_1)+x_2 \g_2(x_1) & a_{12}(x_1,x_2)\\ a_{21}(x_1,x_2) & a_{22}(x_1,x_2) \end{psmallmatrix}$
for some $\g_i \in L^2((0,L)), a_{ij}\in L^2(S)$. Hence, by Lemma~\ref{lem:Q2-bar-def}, 
\begin{equation}\label{eq:lower-second}
\liminf_{t,w\to 0} \frac{1}{t^2} I^{t,w}(y^{t,w})\geq \frac12 \overline{Q}_2\bra{\II}
=\frac{1}{24} \int_S Q_2\bra{\II(x')-\oII(x')}dx' + \frac12 \a, 
\end{equation}
with $\a$ and $\oII$ defined in \eqref{eq:alpha} and \eqref{eq:oII}.
Recall that $\P_2:\mathcal{H}_2\to \mathcal{M}_2$ is the orthogonal projection onto the subspace 
\[
\mathcal{M}_2 = \set{
x'\mapsto \begin{pmatrix}b_{11}(x_1)&b_{12}(x_1)\\b_{12}(x_1)&b_{22}(x')\end{pmatrix} \;:\;
b_{11},b_{12}\in L^2((0,L)), b_{22}\in L^2(S)
}.
\]
Since $\II_{11}$ and $\II_{12}$ depend only on $x_1$, we have $\II-\P_2(\oII) \in \mathcal{M}_2$. Hence
\begin{equation}\label{eq:proj2}
\int_S Q_2\bra{\II(x')-\oII(x')}dx' = \int_S Q_2\bra{\II - \P_2(\oII)} dx' + \int_S Q_2\bra{ (\id-\P_2)(\oII) }dx'.
\end{equation}
Now we employ Jensen's inequality: write
\[
Q_2\bra{\II-\P_2(\oII)} = \sum_{i,j,k,l\in\{1,2\}} \mathbb{M}_{ijkl}\bra{\II-\P_2(\oII)}_{ij}\bra{\II-\P_2(\oII)}_{kl}
\]
for some symmetric fourth-order tensor $\mathbb{M}$, and integrate over $x_2\in(-\frac12,\frac12)$ using Jensen's inequality for the $ijkl=2222$ component (noticing that $s\mapsto s^2$ is convex and  $\mathbb{M}_{2222}=Q_2(\begin{smallmatrix}0&0\\0&1\end{smallmatrix})>0)$ to obtain
\begin{equation}\label{eq:jensen}
\int_{-\frac12}^{\frac12} Q_2\bra{\II-\P_2(\oII)} dx_2\geq Q_2(\fII-\wtfII),
\end{equation}
where $\fII, \wtfII$ are obtained by averaging the $(2,2)$-component of $\II,\P_2(\oII)$ (respectively) over $x_2\in (-\frac12,\frac12)$. 
Explicitly, $\fII$ is the averaged fundamental form associated with $(y,d_2,d_3,\K_3)$, as given in \eqref{eq:fII}, and $\wtfII$ is the reference second fundamental form, defined in Definition \ref{def:effective-excess}. Putting equation \eqref{eq:jensen} back in \eqref{eq:proj2}, and then back in \eqref{eq:lower-second}, we get
 \begin{align*}
 \liminf_{t,w\to 0} \frac{1}{t^2} I^{t,w}(y^{t,w})
  &\geq \frac1{24} \int_0^{L} Q_2\bra{\fII-\wtfII}dx_1 +
  \frac1{24}\int_S Q_2\bra{ (\id-\P_2)(\oII) }dx'
  +\frac12 \a
 \\&=\frac1{24} \int_0^{L} Q_2\bra{\fII-\wtfII}dx_1 +C_{\text{excess}} =I^0(y,d_2,d_3,\K_3),
 \end{align*}
 where we used the definition \eqref{eq:alpha} of $\a$ and the definition \eqref{eq:excess-energy} of the excess energy.
\end{proof}

For the proof of the upper bound part of the theorem, we will need the following approximation lemma:
\begin{lemma}\label{lem:Dtw}
Let $(y,d_2,d_3,\widetilde{K}_3)\in \widetilde{\mathcal{A}}$ and denote 
$R=(y_{,1}\mid d_2\mid d_3)$. Further assume that $y,d_2,d_3,\widetilde{K}_3\in C^\infty([0,L];\mathbb{R}^3)$.
Then for any $\kappa \in C^\infty([0,L])$ we can find sequences 
$d_1^{t,w},d_2^{t,w},d_3^{t,w},\K_3^{t,w} \in C^\infty([0,L];\R^3)$ such that 
\begin{enumerate}[label=(\roman*)]
\item $D^{t,w}\coloneqq(d_1^{t,w}\mid d_2^{t,w}\mid d_3^{t,w})\in \SO(3)$  in $[0,L]$.
\item $D^{t,w} \to R$ in $L^2([0,L];\R^{3\times 3})$   and  $\K_3^{t,w}\to \K_3$ in $L^2([0,L];\R^3)$.
\item For all $t,w$, there holds
\begin{equation}\label{eq:ODE-lem1}
d_1^{t,w}\cdot d_{3,1}^{t,w} =\fII_{11},\qquad d_{2}^{t,w}\cdot d_{3,1}^{t,w} =\fII_{12},\qquad
 d_1^{t,w}\cdot d_{2,1}^{t,w}  = \frac{t}{w} \kappa
\end{equation}
and
\begin{equation}\label{eq:ODE-lem2}
 \widetilde{K}^{t,w}_3 = \fII_{12}\,d_1^{t,w} + \fII_{22}\,d_2^{t,w},
\end{equation}
\end{enumerate}
where
$\fII$ is the averaged second fundamental form associated with $(y,d_2,d_3,\widetilde{K}_3)$, defined as in \eqref{eq:fII}.
\end{lemma}

\begin{proof}
 Let $(d_1^{t,w}\mid d_2^{t,w}\mid d_3^{t,w})\eqqcolon D^{t,w}$ be the unique solution to the linear initial value problem
\[
\begin{cases*}
D_{,1} = D
\renewcommand{\arraystretch}{1.2}\begin{pmatrix}
0 & \frac{t}{w}\kappa&\fII_{11} \\
-\frac{t}{w}\kappa& 0 & \fII_{12} \\
-\fII_{11} &-\fII_{12}&0
\end{pmatrix}\qquad &in $[0,L]$,\\
D(0)=R(0).
\end{cases*}
\]
 The equations \eqref{eq:ODE-lem1} hold immediately by the definition of $D^{t,w}$. Also, because $\frac{t}{w}\kappa$ uniformly converges to zero, we have $D^{t,w}\to R$ (because $R$ is the unique solution of the ODE for $\kappa\equiv 0$). Moreover, $D^{t,w}(x_1)\in \SO(3)$ because the coefficient matrix of the given linear ODE is skew-symmetric. Now define $\K_3^{t,w}$ by the formula \eqref{eq:ODE-lem2}. Because $\K_3 = \fII_{12}y_{,1} +\fII_{22}d_2$ (by the definition of $\wtA$ and $\fII$), it follows that $\K_3^{t,w}\to \K_3$ as claimed.
\end{proof}

\begin{proof}[Proof of Theorem \ref{thm:limit-model} (\textbf{Upper bound}).]
Let $(y,d_2,d_3,\widetilde{K}_3)\in \widetilde{\mathcal{A}}$, and denote 
$R=(y_{,1}\mid d_2\mid d_3)$. Assume first that $y,d_2,d_3,\widetilde{K}_3\in C^\infty([0,L];\mathbb{R}^3)$. Choose arbitrary $\g_1,\g_2\in C^\infty([0,L];\R)$, $a_{12},a_{22}\in C^\infty(\overline{S};\R)$ and $z\in C^\infty_0(\O;\R^3)$. Also, choose arbitrary $f\in C^\infty(\overline{S};\R)$ satisfying $\int_{-\frac12}^{\frac12} f(x_1,x_2)dx_2=0$ for all $x_1$. Let $D^{t,w}=(d_1^{t,w}\mid d_2^{t,w}\mid d_3^{t,w})$ and $\K_3^{t,w}$ be as in Lemma~\ref{lem:Dtw} with $\kappa=\g_2$ (this means that $d_1^{t,w}\cdot d_{2,1}^{t,w} = \frac{t}{w}\g_2$).
Define the following functions:
\begin{align*}
Z(x',x_3)&\coloneqq\int_0^{x_3}z(x',s)ds,\qquad \Gamma_1(x_1)\coloneqq\int_0^{x_1}\g_1(s)d_1^{t,w}(s) ds,
\\ F(x_1,x_2)&\coloneqq \int_{-\frac12}^{x_2} f(x_1,s)ds,\quad A_{12}(x_1,x_2)\coloneqq  \int_{-\frac12}^{x_2} a_{12}(x_1,s)ds,\quad
A_{22}(x_1,x_2)\coloneqq  \int_{-\frac12}^{x_2} a_{22}(x_1,s)ds.
\end{align*}
Also, let $S\in C^\infty(\overline{S};\mathbb{R}^3)$ be the function  
\[S(x_1,x_2)= -\int_{-\frac12}^{x_2}\left[s\,\fII_{22}(x_1)+F(x_1,s)\right]d_3(x_1)\,ds,\]
where $\fII_{22} = d_2\cdot \widetilde{K}_3$ as in \eqref{eq:fII}.
Note that $S_{,3}=0$. Define a sequence $y^{t,w}$ by
\begin{align}\label{eq:finer-y-tw-def}
\begin{split}
y^{t,w} = \int_0^{x_1} d_1^{t,w}(s)ds&+ wx_2d_2^{t,w}+t\bra{x_3d_3^{t,w}+\Gamma_1}
\\&+wt\sqbra{x_3\bra{x_2\widetilde{K}_3+Fd_2^{t,w}}+2A_{12}d_1^{t,w}+A_{22}d_2^{t,w}}
\\&+w^2 S +t^2 D^{t,w}Z.
\end{split}
\end{align}
Differentiating, we find that
\begin{equation}\left\{
\begin{aligned}
y^{t,w}_{,1} &= d_1^{t,w} + w x_2 d_{2,1}^{t,w} + t \bra{x_3 d_{3,1}^{t,w}+\g_1d_1^{t,w}} 
+ o(t),
\\\frac{1}{w}y^{t,w}_{,2} &= d_2^{t,w} + t \sqbra{x_3\bra{ \K_3^{t,w}+f d_2^{t,w}}+2a_{12}d_1^{t,w}+a_{22}d_2^{t,w}} + w S_{,2}  + o(t),
\\\frac{1}{t}y^{t,w}_{,3} &= d_3^{t,w} +w \bra{x_2 \K_3^{t,w} +F d_2^{t,w}}+ t D^{t,w}z 
\end{aligned}\right.
\end{equation}
and therefore
\begin{align}\label{eq:Dy-tw}
\begin{split}
\D_{t,w} y^{t,w} =
D^{t,w} &+ t\bra{x_3 d_{3,1}^{t,w}+\g_1d_1^{t,w}\,\mid\,x_3(\K_3^{t,w}+fd_2^{t,w})+ 2a_{12}d_1^{t,w}+a_{22}d_2^{t,w}\,\mid\,D^{t,w}z}
\\&+ w \bra{x_2 d_{2,1}^{t,w}\,\mid\,S_{,2}\,\mid\,x_2 \K_3^{t,w}+\widetilde{F}d_2^{t,w}}+o(t),
\end{split}
\end{align}
where $o(t)$ denotes a function in $L^2(\O;\mathbb{R}^{3\times 3})$ such that $o(t)/t\to 0$ as $t,w\to 0$ (in the narrow sense).
Before we continue, notice that  $y^{t,w}\to y$ in $W^{1,2}$ (up to translating $y^{t,w}$),  $\D_{t,w}y^{t,w}\to R$ in $L^2$, and \eqref{eq:gen-y23-tw-convergence} holds for $K_3\coloneqq \widetilde{K}_3+fd_2$, because of the fact that $\frac{1}{tw}y^{t,w}_{,32} = \widetilde{K}_3^{t,w}+fd_2^{t,w}+O(t/w)$  in $L^2(\O;\mathbb{R}^3)$ and that $\int_{-\frac12}^{\frac12}f(x_1,x_2)dx_2=0$ by choice of $f$. 
Thus it remains to prove the convergence of energies.
Multiplying equation \eqref{eq:Dy-tw} from the left by $(D^{t,w})^T$, using $S_{,2}=-(x_2\fII_{22}+F)d_3$, and recalling the relations \eqref{eq:ODE-lem1} and \eqref{eq:ODE-lem2}, we get
\begin{align}\label{eq:gen-A-def}
\begin{split}
(D^{t,w})^T\D_{t,w} &y^{t,w} =\\
&I+ t \;
\underbrace{\begin{pNiceArray}{w{c}{2.2cm}w{c}{2.2cm}|c}
	\Block{2-2}{{\phantom{w}\begin{psmallmatrix}\g_1+x_2\g_2&2a_{12}
			\\0&a_{22}\end{psmallmatrix}}+ x_3 \sqbra{\fII+ \begin{psmallmatrix}0&0\\0&f\end{psmallmatrix}}\phantom{w}}& & \\ & &z\\ \cline{1-2} 0&0&\end{pNiceArray}
}_{\displaystyle \coloneqq A}
+w\; \underbrace{\begin{pmatrix}0&0&x_2 \fII_{12}\\0&0&x_2\fII_{22}+F\\-x_2\fII_{12}&-x_2\fII_{22}-F&0\end{pmatrix}}_{\displaystyle \coloneqq H}
+o(t).
\end{split}
\end{align}
By frame indifference of $W$ and by polar decomposition, for any $M\in \mathbb{R}^{3\times 3}$ with positive determinant we have 
$W(M)=W(\sqrt{M^TM})$. Since $(D^{t,w})^T\D_{t,w}y^{t,w}$ and $I+ t B$ both have positive determinant for small enough $t,w$ (both tend to $I$ as $t,w\to0$), we can apply this to $M\coloneqq (D^{t,w})^T\D_{t,w}y^{t,w}\bra{I+ t B}$. By skew-symmetry of $H$, and since $w^2$ is $o(t)$ (where here and below we use $o(t)$ in a uniform sense),
  we get
\begin{align*}
M^T M &= \bra{I+ t B^T}\bra{I+tA^T-wH+o(t)}\bra{I+tA+wH+o(t)}\bra{I+t B}
\\&=\bra{I+ t B^T}\bra{I+2 t \Sym A +o(t)}\bra{I+ t B}
\\&= I +2t \Sym A +2t B +o(t),
\end{align*}
so that $\sqrt{M^TM} = I+t (\Sym A+ B) +o(t)$.  Hence
\begin{align*}
W\bra{\D_{t,w} y^{t,w} \bra{I+ t B}}
=  W\bra{I+t\bra{ \Sym A +  B} + o(t)}.
\end{align*}
For small enough $t,w$, the matrix $I+t\bra{ \Sym A +   B} +o(t)$ belongs to a neighbourhood of $I$ in which $W$ is of class $C^2$.  Taylor expansion gives
\begin{align*}
W\bra{\D_{t,w} y^{t,w} \bra{I+ t B}} &= \frac12 Q_3\big(t\bra{ \Sym A +  B} +o(t)\big) + o(t^2)
\\&= \frac12 t^2 Q_3\bra{\Sym A+  B}+o(t^2),
\end{align*}
whence
\[\frac{1}{t^2}W\bra{\D_{t,w} y^{t,w} \bra{I+ t B}} \xrightarrow[t,w\to 0\;\text{narrowly}]{} \frac12 Q_3\bra{\Sym A+  B} \quad\text{pointwise in}\; \O,\] 
and
\[\frac{1}{t^2}\abs{W\bra{\D_{t,w} y^{t,w} \bra{I+ t B}}} \leq  C\bra{\abs{\Sym A+ B}^2+1}.\]
Invoking the dominated convergence theorem, we obtain
\begin{equation}\label{eq:finer-to-Q3}
\frac{1}{t^2}I^{t,w}(y^{t,w}) \xrightarrow[t,w\to 0]{}  \frac12\int_\O Q_3\bra{\Sym A+ B}dx=\frac12\int_\O Q_3\bra{ \begin{pNiceArray}{w{c}{1.3cm}w{c}{1.3cm}|c}
	\Block{2-2}{ g + x_3 \sqbra{\fII+ \begin{psmallmatrix}0&0\\0&f\end{psmallmatrix}}}& & \\ & &z\\ \cline{1-2} 0&0&\end{pNiceArray}+  B}dx,
\end{equation}
where for the last equality we used the definition \eqref{eq:gen-A-def} of $A$, and where we defined $g\coloneqq\begin{psmallmatrix}\g_1+x_2\g_2&a_{12}\\a_{12}&a_{22}\end{psmallmatrix}$. Since $C_0^{\infty}(\O;\mathbb{R}^3)$ is dense in $L^2(\O;\mathbb{R}^3)$, by a standard diagonalization argument (see Appendix~\ref{appendix:diag}) choosing $z=z^{t,w}$ suitably leads to $y^{t,w}$ such that
\[\frac{1}{t^2}I^{t,w}(y^{t,w}) \xrightarrow[t,w\to 0\;\text{narrowly}]{} \frac12 \int_\O Q_2\bra{g+ x_3 \sqbra{\fII+ \begin{psmallmatrix}0&0\\0&f\end{psmallmatrix}} +\check{B}}dx.\]
Recall from \eqref{eq:P3B-kg} that
$\P_3(\check{B})= \widetilde{A} -x_3 \oII$ where
$\widetilde{A}=\begin{psmallmatrix}\kappa_0+x_2\kappa_g&\tilde{a}_{12}\\\tilde{a}_{12}&\tilde{a}_{22}\end{psmallmatrix}$
for some $\kappa_0, \kappa_g\in L^2((0,L)), \tilde{a}_{ij}\in L^2(S)$ (where $\kappa_g$ is the reference geodesic curvature).
Therefore we can approximate $-\widetilde{A}$ using $g$: another diagonalization argument shows that an appropriate choice of $g=g^{t,w}$ gives $y^{t,w}$ for which
\begin{align*}
\frac{1}{t^2} I^{t,w}(y^{t,w}) &\xrightarrow[t,w\to 0]{} \frac12 \int_\O  
Q_2\bra{-\widetilde{A}+ x_3 \sqbra{\fII+ \begin{psmallmatrix}0&0\\0&f\end{psmallmatrix}} +\check{B}}
\\&= \frac12 \int_\O Q_2\bra{ x_3 \sqbra{\fII+ \begin{psmallmatrix}0&0\\0&f\end{psmallmatrix}-\oII} + (\id-\P_3)(\check{B}) }dx
\\&= \frac{1}{24} \int_S Q_2\bra{\fII+ \begin{psmallmatrix}0&0\\0&f\end{psmallmatrix}-\oII }dx' +\frac12 \a
\end{align*}
with $\a$ and $\oII$ defined in \eqref{eq:alpha} and \eqref{eq:oII}. 
Now consider the projection $\P_2(\oII)$ of $\oII$ (cf. \S~\ref{sec:excess-energy}).
Because $\wtfII = \int_{-\frac12}^{\frac12} P_2(\oII) dx_2$, it follows that $\P_2(\oII)-\wtfII$ has the form $\begin{psmallmatrix}0&0\\0&h\end{psmallmatrix}$ where $h\in L^2(S)$ satisfies 
$\int_{-\frac12}^{\frac12}h(x_1,x_2)dx_2=0$ for all $x_1\in (0,L)$. Therefore, we may approximate $\P_2(\oII)-\wtfII$ using $\begin{psmallmatrix}0&0\\0&f\end{psmallmatrix}$ by appropriately choosing $f=f^{t,w}$, obtaining
\begin{align*}
\frac{1}{t^2} I^{t,w}(y^{t,w}) &\xrightarrow[t,w\to 0]{} \frac1{24} \int_S Q_2\bra{\fII+ \P_2(\oII)-\wtfII-\oII}dx' + \frac12 \a
\\&=\frac{1}{24} \int_S Q_2\bra{\fII-\wtfII-(\id-\P_2)(\oII)}dx'+\frac12 \a
\\&=\frac{1}{24} \int_0^{L} Q_2\bra{\fII-\wtfII}dx_1+\frac{1}{24} \int_S Q_2\bra{(\id-\P_2)(\oII) }  dx' +\frac12 \a
\\&=\frac{1}{24} \int_0^{L} Q_2\bra{\fII-\wtfII}dx_1+C_{\text{excess}}
=I^0(y_{,1},d_2,d_3,\widetilde{K}_3),
\end{align*}
where $C_{\text{excess}}$ is the excess energy defined in \eqref{eq:excess-energy}. This proves the smooth case; consider now the general case. As shown in \cite[Proof of Theorem 5.1(ii)]{fmp12}, since $(y,d_2,d_3)\in \mathcal{A}$, we can find sequences $y^n,d_2^n,d_3^n \in C^\infty([0,L];\mathbb{R}^3)$ such that $(y^n,d_2^n,d_3^n)\in\mathcal{A}$ for all $n$, and moreover $y^n$ converges to $y$ strongly in $W^{2,2}(0,L);\mathbb{R}^3)$ while $d_2^n,d_3^n$ converge (respectively) to $d_2,d_3$ strongly in $W^{1,2}((0,L);\mathbb{R}^3)$. By the definition of $\widetilde{\mathcal{A}}$, there exists a function $\widetilde{\gamma}\in L^2((0,L)) $ such that 
\[\widetilde{K}_3 = (d_2\cdot d_{3,1})y_{,1} + \widetilde{\gamma}\, d_2.\]
Let $\widetilde{\gamma}^n\in C^\infty([0,L])$ be a sequence such that $\widetilde{\gamma}^n\to \widetilde{\gamma}$ in $L^2((0,L))$, and define a sequence $\widetilde{K}_3^n\in W^{1,2}((0,L);\mathbb{R}^3)$ by 
\[ \widetilde{K}_3^n \coloneqq (d_2^n\cdot d_{3,1}^n)y_{,1}^n + \widetilde{\gamma}^n\, d_2^n.\]
By  Hölder's inequality, because $d_2^n\to d_2$ in $W^{1,2}((0,L);\mathbb{R}^3)$, we have $d_2^n\to d_2 $ in $L^\infty((0,L);\mathbb{R}^3)$. Hence
\begin{align*}
\norm{\widetilde{\gamma}^nd_2^n-\widetilde{\gamma}d_2}_{L^2((0,L);\mathbb{R}^3)}
&\leq  \norm{\widetilde{\gamma}^nd_2^n-\widetilde{\gamma}d_2^n}_{L^2((0,L);\mathbb{R}^3)} +
\norm{\widetilde{\gamma}d_2^n-\widetilde{\gamma}d_2}_{L^2((0,L);\mathbb{R}^3)}
\xrightarrow[n\to\infty]{} 0,
\end{align*}
so that $\widetilde{\gamma}^nd_2^n\to\widetilde{\gamma}d_2$ strongly in $L^2((0,L);\mathbb{R}^3)$. Similarly it is shown that
$(d_2^n\cdot d_{3,1}^n)y_{,1}^n\to (d_2\cdot d_{3,1})y_{,1}$ strongly in $L^2((0,L);\mathbb{R}^3)$,  so that altogether $\widetilde{K}_3^n\to \widetilde{K}_3$ strongly in $L^2((0,L);\mathbb{R}^3)$. Thus, we have obtained a sequence of quadruplets
$(y^n,d_2^n,d_3^n,\widetilde{K}_3^n)\in\widetilde{\mathcal{A}}$ which satisfy $y^n,d_2^n,d_3^n,\widetilde{K}_3^n\in C^\infty([0,L];\mathbb{R}^3)$, and moreover $y^n\to y$ in $W^{2,2}((0,L);\R^3)$, while $d_2^n,d_3^n$ converge to $d_2,d_3$ (respectively) strongly in $W^{1,2}((0,L);\mathbb{R}^3)$, and $\widetilde{K}_3^n\to\widetilde{K}_3$ strongly in $L^2((0,L);\mathbb{R}^3)$. By the first part of the proof, for each $n\in \mathbb{N}$ we can find a sequence $y^{t,w,n}$ such that $y^{t,w,n}\xrightarrow[t,w\to 0]{} y^n$ strongly in $W^{1,2}((0,L);\mathbb{R}^3)$, $\D_{t,w}y^{t,w,n}\xrightarrow[t,w\to 0]{} (y_1^n|d_2^n|d_3^n)$ strongly in 
$L^2(\O;\mathbb{R}^{3\times 3})$, $\frac{1}{tw}y_{,23}^{t,w,n} \xrightharpoonup[t,w\to 0]{} \widetilde{K}_3^n$ weakly in $W^{-1,2}((0,L);\mathbb{R}^3)$, and

\[
\frac{1}{t^2}I^{t,w}(y^{t,w,n}) \xrightarrow[t,w\to 0]{}\frac1{24}
\int_0^L Q_2\bra{\begin{pmatrix}y_{,1}^n\cdot d_{3,1}^n & d_2^n\cdot d_{3,1}^n\\d_{2}^n\cdot d_{3,1}^n& d_2^n\cdot \widetilde{K}_3^n\phantom{\bigg|}\end{pmatrix}-\wtfII}dx_1 + C_{\text{excess}}.
\]
By another standard diagonalization argument, there exists a sequence $n_{t,w}\to \infty$ such that the sequence $y^{t,w}\coloneqq y^{t,w,n_{t,w}}$ has all the required properties.
\end{proof}

\newpage
\section{Wide Ribbons}\label{sec:wide-ribbons}
In this section we develop a limiting theory of wide ribbons. We first establish a plate theory for $\Glim_{t\to 0} t^{-2} I^{t,w}$, and then analyse the limit and behaviour 
 of the resulting plate energy as $w\to 0$.   Non-Euclidean Kirchhoff plate theories have been widely studied; the case considered here is very similar to \cite{bls16,ll20,pg22}, the last of which offering the required generality of prestrain for our case. The wide ribbon analysis largely follows the arguments in \cite{fhmp16b, mm25}. Due to technical difficulties, we are able to prove the full wide ribbon limit theory only under an additional assumption on the prestrain $B$; namely that $\P_2(\oII)$ is independent of $x_2$, where $\P_2$ is the orthogonal projection defined in \S\ref{sec:excess-energy}. 

We begin by stating the compactness and $\Gamma$-convergence results of the relevant plate theory.
The proof of the $\Gamma$-convergence is easily adapted from the results in \cite{pg22} --- we outline this derivation in Appendix~\ref{sec:appendix}.
Here we denote $\D_w \coloneqq \bra{\p_1 \mid \frac{1}{w}\p_2 }$.

\begin{theorem}[Plates, compactness] \label{thm:plates-compactness}
Fix $w>0$, and let $(y^t)\subseteq W^{1,2}(\O;\R^3)$ be a sequence with finite  energy, i.e., $I^{t,w}(y^t) \leq C t^2$. Then (modulo a subsequence and translations) $y^{t}\to y$ in $W^{1,2}$ and $\D_{t,w} y^t \to\bra{\D_w y \mid \nu_y}$ in $L^2$, where $y\in W^{2,2}(S;\R^3)$ is a \emph{rescaled isometric immersion}, i.e., $\D_w y^\top \D_w y = I_{2\times 2}$, and $\nu_y$ is the normal to $y$.
\end{theorem}

\begin{theorem}[Plates, $\Gamma$-convergence]\label{thm:plates-gamma-convergence}
	\leavevmode
\begin{enumerate}
	\item[(i)] \textbf{Lower Bound}:
Under the assumptions of Theorem~\ref{thm:plates-compactness}, we have
\[
\liminf_{t\to 0} \frac{1}{t^2} I^{t,w}(y^t) \geq J^w(y) \coloneqq 
\frac{1}{24}\int_S Q_2\bra{\II^w_y (x') - \overline{\II}(x')}dx'+C_{\text{excess}}^{(0)},
\]
where 
$\II^w_y\coloneqq -\D_wy^\top \D_w \nu_y$ 
 is the (rescaled) second fundamental form of $y$, and the reference form $\overline{\II}$ and excess energy $C_{\text{excess}}^{(0)}$ are given by
\begin{equation}\label{eq:plates-gamma-convergence-reference}
\overline{\II}(x')=-12\int_{-\frac12}^{\frac12}x_3 \check{B}(x',x_3)dx_3,\qquad C_{\text{excess}}^{(0)}=\frac12 \int_\O Q_2\bra{ (\id-\P)(\check{B}(x)) }dx ,
\end{equation}
where $\P$ is the orthogonal projection from $L^2(\O;\R^{2\times 2}_\sym)$ onto the space of $x_3$-affine functions.
\item[(ii)] \textbf{Recovery Sequence}: For any rescaled isometric immersion $y\in W^{2,2}(S;\R^3)$, there exists a sequence $(y^t)\subseteq W^{1,2}(\O;\R^3)$ with such that $y^{t}\to y$ in $W^{1,2}$, $\D_{t,w} y^t \to\bra{\D_w y \mid \nu_y}$ in $L^2$, and
\[
\lim_{t\to 0} \frac{1}{t^2}I^{t,w}(y^t) =J^w(y).
\]
\end{enumerate}
\end{theorem}

The following compactness result for wide ribbons is taken directly from the more general statement in \cite[Theorem~5.5]{mm25}, while a detailed proof may be found in \cite[Lemma~2]{fhmp16b}. The main result in this section --- the $\Gamma$-convergence in Theorem~\ref{thm:wide-ribbons-gamma-convergence} --- is adapted from \cite[Theorem~5.6]{mm25}.

\begin{theorem}[Wide ribbons, compactness]\label{thm:wide-ribbons-compactness}
Let $(y_w)$ be a sequence of rescaled isometric immersions in $W^{2,2}(S;\R^3)$ such that $J^w(y_w) \leq C$ for every $w>0$. Then, modulo a subsequence and translations, we have
\begin{equation}\label{eq:wide-ribbons-compactness}
y_w \xwto{W^{2,2}} y, \quad
\D_w y_w \xwto{W^{1,2}} (y_{,1}\mid d_2), \quad
\II^w_{y_w} \xwto{L^2} \II\coloneqq \begin{pmatrix}y_{,1}\cdot d_3 & d_{2,1}\cdot d_3\\ d_{2,1}\cdot d_3 & -\gamma\end{pmatrix},
\end{equation}
where $y\in W^{2,2}((0,L);\R^3)$, $R\coloneqq (y_{,1}\mid d_2\mid d_3)\in W^{1,2}((0,L); \SO(3))$ with $y_{,1}\cdot d_{2,1}=0$, and $\gamma \in L^2(S)$.
\end{theorem}

\begin{theorem}[Wide ribbons, $\Gamma$-convergence]\label{thm:wide-ribbons-gamma-convergence}
	\leavevmode
\begin{enumerate}
	\item[(i)] \textbf{Lower bound}:
	Under the assumptions of Theorem~\ref{thm:wide-ribbons-compactness}, we have
	\[
	\liminf_{w\to 0} J^w(y_w) \geq J^0(\fII)\coloneqq
	\frac{1}{24}\int_0^L \bra{
		Q_2(\fII-\wtfII)+\alpha_Q^+(\det \fII)^+ + \alpha_Q^-(\det \fII)^-
	}dx_1 +C_{\text{excess}},
	\]
where $\fII \coloneqq \int_{-\frac12}^{\frac12}\II dx_2$ is the second fundamental form along the midline, $\wtfII$ is the reference second fundamental form as in Definition~\ref{def:effective-excess}, and the excess energy $C_{\text{excess}}\geq0$ is given by
\begin{equation}\label{eq:wide-ribbon-excess-energy}
C_{\text{excess}} \coloneqq
\frac{1}{24} \int_S Q_2\bra{(\id-\P_2)(\oII) }  dx' + 
\frac12 \int_\O Q_2\bra{ (\id-\P)(\check{B}) }  dx,
\end{equation}	
where $\P_2$ is the orthogonal projection defined in \S\ref{sec:excess-energy}, and $\P$ is the orthogonal projection from $L^2(\O;\R^{2\times 2}_\sym)$ onto the space of $x_3$-affine functions.

Here, the positive constants $\alpha_Q^\pm$ are given by
\[
\alpha_Q^\pm \coloneqq \sup\set{ \alpha >0 \;:\;
	\forall M\in \R^{2\times 2}_\sym,\quad 
	 Q_2(M)\pm \a \det M \geq 0  }.
\]
	\item[(ii)] \textbf{Recovery sequence}: Assume that $\P_2(\oII) = \wtfII$.
For any $y\in W^{2,2}((0,L);\R^3)$, $R=(y_{,1}\mid d_2\mid d_3)\in W^{1,2}((0,L);\SO(3))$ with $y_{,1}\cdot d_{2,1}=0$, and $\gamma\in L^2((0,L))$,  there exists a sequence $(y_w)$ of rescaled isometric immersions in $W^{2,2}(S;\R^3)$ such that \eqref{eq:wide-ribbons-compactness} holds, and moreover
\[
\lim_{w\to 0} J^w(y_w)=J^0(\fII)
\]
where $\fII=\II$ is defined as in \eqref{eq:wide-ribbons-compactness}.

\end{enumerate}
\end{theorem}
We split the proof into two parts.
\begin{proof}[Proof of Theorem~\ref{thm:wide-ribbons-gamma-convergence} (\textbf{Lower bound}).]
Notice that
\begin{align*}
	\liminf_{w\to 0} J^w(y_w) 
	&= 	\liminf_{w\to 0} \frac{1}{24}\int_S Q_2\bra{\II^w_y (x') - \overline{\II}(x')}dx'+C_{\text{excess}}^{(0)}
	\\&= \liminf_{w\to 0} \frac{1}{24}\int_S \bra{Q_2(\II^w_{y_w})+Q_2(\overline{\II})-2 Q_2(\II^w_{y_w},\overline{\II}) }dx'+C_{\text{excess}}^{(0)}.
\end{align*}
Because $y_w$ is a $W^{2,2}$-isometric immersion, by the Gauss equations we have $\det \II^w_{y_w}=0$ a.e.\ on $S$. Hence, by \cite[
Proposition 9]{fhmp16b} and by the established convergence \eqref{eq:wide-ribbons-compactness} of $\II^w_{y_w}$, it follows that  
 \begin{align*}
 \liminf_{w\to 0} J^w(y_w) 
 	&\geq   \frac{1}{24}\int_S \bra{Q_2(\II) +\alpha_Q^+(\det \II)^+ + \alpha_Q^-(\det \II)^-+Q_2(\overline{\II})-2 Q_2(\II,\overline{\II}) }dx'+C_{\text{excess}}^{(0)}.
 \\&= 
 \frac{1}{24}\int_S \bra{
 	Q_2(\II-\overline{\II}) +\alpha_Q^+(\det \II)^+ + \alpha_Q^-(\det \II)^-
 }dx' +C_{\text{excess}}^{(0)}.
\end{align*}

Similar to the proof of Theorem~\ref{thm:limit-model}, notice that since $\II_{11}$ and $\II_{12}$ depend only on $x_1$, we have $\II-\P_2(\oII) \in \mathcal{M}_2$ (where $\P_2$ is the projection onto matrix-fields in $L^2(S;\R^{2\times 2}_\sym)$ whose first two components are constant in $x_2$, cf. \S\ref{sec:excess-energy}). Hence
\begin{equation*}
\liminf_{w\to 0} J^w(y_w) \geq 
\frac{1}{24}\int_S \bra{
	Q_2(\II-\P_2(\overline{\II})) +\alpha_Q^+(\det \II)^+ + \alpha_Q^-(\det \II)^-
}dx' +C_{\text{excess}},
\end{equation*}
where now the excess energy $C_{\text{excess}}$ is given by \eqref{eq:wide-ribbon-excess-energy}.
By Jensen's inequality and by the positive-definiteness of $Q_2$, together with the fact that the determinant of a matrix is linear in the (2,2)-component, the claim follows.
\end{proof}

As mentioned, we prove the second part of the theorem only under the additional assumption that $\P_2(\oII)$ is independent of $x_2$, i.e., $\P_2(\oII) = \wtfII$.
In that case, the proof is nearly identical to the proof given in \cite[Proof of Theorem~5.6]{mm25}, and so we will omit some details in the proof displayed below.
\begin{proof}[Proof of Theorem~\ref{thm:wide-ribbons-gamma-convergence} (\textbf{Recovery sequence}).] By \cite[Proposition~9]{fhmp16b}, there exists a sequence   $(M^\delta)\subseteq L^2((0,L);\R^{2\times 2}_\sym)$ such that $M^\delta \xwto{L^2} \fII$ as $\delta\to 0$, $\det M^\delta =0$ for every $\delta >0$, and
 \begin{align*}
\lim_{\delta\to 0} \frac{1}{24}\int_S & Q_2(M^\delta-\P_2(\overline{\II})) dx' +C_{\text{excess}}
\\&=
\frac{1}{24}\int_S \bra{
	Q_2(\fII-\P_2(\overline{\II})) +\alpha_Q^+(\det \fII)^+ + \alpha_Q^-(\det \fII)^-
}dx' +C_{\text{excess}}
\\&=
\frac{1}{24}\int_0^L \bra{
	Q_2(\fII-\wtfII)+\alpha_Q^+(\det \fII)^+ + \alpha_Q^-(\det \fII)^-
}dx_1 +C_{\text{excess}} \eqqcolon J^0(\fII),
 \end{align*}
where the second equality follows by the assumption that $\P_2(\oII) = \wtfII$.  
By the exact same argument as in \cite{mm25} (which also offers  some intuition on the construction), we obtain a sequence of rescaled isometric immersions $(y_w^\delta)\subseteq W^{2,2}(S;\R^3)$ satisfying
\[
y_w^\delta \xto{W^{2,2}} y^\delta,\quad
\D_w y_w^\delta \xto{W^{1,2}} (R^\delta e_1, R^\delta e_2),\quad \II^w_{y_w^\delta} \xto{L^2} M^\delta,
\]	
as $w\to 0$, where $R^\delta\in W^{1,2}((0,L);\SO(3))$ is the frame satisfying the Darboux frame equation
\[
\begin{cases*}
R^\delta_{,1} = R^\delta
\renewcommand{\arraystretch}{1.2}\begin{pmatrix}
0 & 0&-M^\delta_{11} \\
0& 0 & -M^\delta_{12} \\
M^\delta_{11} &M^\delta_{12}&0
\end{pmatrix}\qquad &a.e. in $[0,L]$,\\
R^\delta(0)=R(0),
\end{cases*}
\]
and where $y^\delta(x_1)\coloneqq y(0)+\int_0^{x_1}R^\delta(t)e_1 dt$. By dominated convergence we have
\[
\lim_{w\to 0}J^w(y_w^\delta) = \frac{1}{24}\int_S  Q_2(M^\delta-\P_2(\overline{\II})) dx' +C_{\text{excess}},
\]
and so the claim follows by a diagonal argument.
\end{proof}

\newpage
\begin{appendices}

\section{Diagonalization Arguments}\label{appendix:diag}
In this appendix we provide formal justification for the diagonalization arguments employed in the proof of \ref{thm:limit-model}, such as the diagonalization argument involving $z=z^{t,w}$, or the one  involving $y^n,d_2^n,d_3^n,\widetilde{K}_3^n$. They stem from the following basic lemma.
\begin{lemma}\label{lem:diagonalization} Let $(X,d)$ be a metric space. Suppose that we are given a family $\mathcal{F}=(x_{t,n})\subseteq X$, indexed over $t>0$ and $n\in \mathbb{N}$, such that the following limits exists:
	\[\forall \,n\in \mathbb{N},\qquad \lim_{t\to 0} x_{t,n} \eqqcolon x_n\in X\]
	and
	\[\lim_{n\to \infty} x_n \eqqcolon x\in X.\]
Then there exists a subfamily $(x_t)\subseteq \mathcal{F}$, indexed over $t>0$, such that $\lim_{t\to 0} x_t = x$.
\end{lemma}
\begin{proof}
For each $n\in \mathbb{N}$, let $t(n)>0$ be a number such that 
$d(x_{t,n},\;x_n)< \frac{1}{n}$ for all $t\leq t(n)$ (there exists such a number, since $x_{t,n}\xrightarrow[t\to 0]{} x_n$). Inductively, define $t'(n)=\min\left\{t(n),\frac12 t'(n-1)\right\}$, with $t'(2)=t(2)$. Notice that $t'(n)$ is a strictly decreasing null-sequence, so we can define $x_t$ for all $t>0$ as follows:
	\[
	x_t \coloneqq \begin{cases*}
	x_{t,1} & when $t > t'(2)$,\\
	x_{t,n}\quad&when $t'(n) \geq t >t'(n+1)$.
	\end{cases*}
	\]
Indeed, $(x_t)$ is a subfamily of $\mathcal{F}$. Moreover we have $x_t\to x$: given $\eps>0$, let $N\in \mathbb{N}$ be large enough enough so that both $\frac{1}{N} <\eps$, and $d(x,\;x_{n})<\eps$ for all $n\geq N$. 
We shall prove that $|x_t-x|<2\eps$ for all $t<t(N)$. Let $t<t(N)$, and let $n\geq N$ be such that $t(n)\geq t > t(n+1)$. Then
\begin{align*}
d(x_t,\; x )= d(x_{t,n},\;x) \leq d(x_{t,n},\;x_{n})+d(x,\;x_{n}) 
\leq \frac{1}{n} + \eps \leq \frac{1}{N}+\eps \leq 2\eps,
\end{align*}
where the inequality $d(x_{t,n},\;x_{n})\leq \frac{1}{n}$ holds since $t\leq t(n)$.
\end{proof}

As an example, let us provide the full details for the diagonalization argument involving $z=z^{t,w}$.
For each $z\in L^2(\O;\mathbb{R}^3)$, let
\[F(z)= \frac12\int_\O Q_3\bra{ \begin{pNiceArray}{w{c}{1.3cm}w{c}{1.3cm}|c}
	\Block{2-2}{ g + x_3 \sqbra{\fII+ \begin{psmallmatrix}0&0\\0&f\end{psmallmatrix}}}& & \\ & &z\\ \cline{1-2} 0&0&\end{pNiceArray}+  B}dx.\]
By this point, as shown in \eqref{eq:finer-to-Q3}, we know that for each $z\in C_0^\infty(\O;\mathbb{R}^3)$ we have
\begin{equation}\label{eq:App-A-Itw-to-F}
J^{t,w}(z)\coloneqq \frac{1}{t^2}I^{t,w}(y^{t,w}[z]) \xrightarrow[t,w\to 0]{} F(z),
\end{equation}
where $y^{t,w}[z]$ is defined as in \eqref{eq:finer-y-tw-def}, with the dependence on $z$ being highlighted. Now, for each $x\in \O$, there corresponds a unique $z_{\text{min}}$ for which 
\[Q_2(g+x_3 \sqbra{\fII+ \begin{psmallmatrix}0&0\\0&f\end{psmallmatrix}}+\Sym \check{B})=Q_3\bra{ \begin{pNiceArray}{w{c}{1.3cm}w{c}{1.3cm}|c}
	\Block{2-2}{ g + x_3 \sqbra{\fII+ \begin{psmallmatrix}0&0\\0&f\end{psmallmatrix}}}& & \\ & &z_{\text{min}}\\ \cline{1-2} 0&0&\end{pNiceArray}+  B}\]
(see Lemma \ref{lem:bigQ2-is-quadratic} for more details; this follows from the convexity of $Q_3$). This defines a function $z_{\text{min}}: \O \to \mathbb{R}^3$; furthermore we have $z_{\text{min}}\in L^2(\O;\mathbb{R}^3)$ as seen from the coercivity of $Q_3$ (equation \eqref{eq:Q3-coercivity}). Now, by the definition of $z_{\text{min}}$, we have
\begin{equation}
\frac12 \int_\O Q_2\bra{g+ x_3 \sqbra{\fII+ \begin{psmallmatrix}0&0\\0&f\end{psmallmatrix}} +\check{B}}dx =  F(z_{\text{min}}) .
\end{equation}
We wish to find a sequence $z^{t,w}\in C_0^\infty(\O;\mathbb{R}^3)$ for which  
\begin{equation}\label{eq:App-A-1-wanted}
J^{t,w}(z^{t,w}) \xrightarrow[t,w\to 0]{} F(z_{\text{min}}).
\end{equation}
Since $C_0^{\infty}(\O;\mathbb{R}^3)$ is dense in $L^2(\O;\mathbb{R}^3)$, we can find a sequence $z^n\in C_0^{\infty}(\O;\mathbb{R}^3)$ such that $z^n \to z_{\text{min}}$ in $L^2(\O;\mathbb{R}^3) $ as $n\to \infty$; since $Q_3$ is a quadratic form, this implies that \[F(z^n)\xrightarrow[n\to\infty]{} F(z_{\text{min}}).\]  By equation \eqref{eq:App-A-Itw-to-F}, for each $n\in \mathbb{N}$ we have 
$J^{t,w}(z^n) \xrightarrow[t,w\to 0]{} F(z^n)$. A straightforward application of lemma \ref{lem:diagonalization} gives us our required $z^{t,w}$.

\newpage
\section{Plates with Rough Prestrain}\label{sec:appendix}
In this appendix we explicitly show how to obtain Theorem~\ref{thm:plates-gamma-convergence} from the (differently phrased) result in \cite{pg22}. 
Recall that $\D_t \coloneqq \bra{\p_1 \mid \p_2 \mid \frac{1}{t} \p_3}$ while 
$\D_{t,w} \coloneqq \bra{\p_1 \mid \frac{1}{w}\p_2 \mid \frac{1}{t} \p_3}$.
\begin{proof}[Proof of Theorem~\ref{thm:plates-gamma-convergence}.]
	Let $(y^t)\subseteq W^{1,2}(\O;\R^3)$ be a sequence with $I^{t,w}(y^t)\leq C t^2$. Define $(v^t)\in W^{1,2}(\O^w;\R^3)$ by $v^t(x)=y^t\bra{x_1,\frac{x_2}{w},x_3}$, where $\O^w \coloneqq \bra{0,L}\times \bra{-\frac{w}{2},\frac{w}{2}} \times \bra{-\frac12,\frac12} $. Notice that
	\begin{align*}
	I^{t,w}(y^t) &= \int_\O W\bra{\D_{t,w}y^t(x)(I+t B(x))}dx
	\\&\stackrel{x_2\mapsto x_2/w}{=} \frac{1}{w}\int_{\O^w} W \bra{ \D_t v^t(x)(I +t B_0(x))  }dx \eqqcolon \E^t(v^t),
	\end{align*}
	where $B_0(x)\coloneqq B(x_1,x_2/w,x_3)$, and $\E^t$ is the energy functional as considered in \cite{pg22}. Therefore $\E^t(v^t)\leq C t^2$, so as shown in \cite{pg22},  the sequence $\frac{1}{t^2}\E^t(v^t)$ converges (in the appropriate $\Gamma$-sense) to the limiting functional
	\begin{equation}\label{eq:appendix-functional}
	\widetilde{J}^w(v) = \frac12 \cdot\frac{1}{w}\int_{\O^w} Q_2\bra{x_3 \D v^\top \D \nu_v +\check{B}_0(x)-\int_{-\frac12}^{\frac12}\check{B}_0(x',t)dt}dx
	\end{equation}
	(the above form of the limiting functional is simpler than the one appearing in \cite{pg22} because we consider the case of homogenous material, i.e., $W$ depends only on the gradient). Here, $v$ is the $W^{1,2}$-limit of $v^t$, in a notion of compactness analogous to that described in Theorem~\ref{thm:plates-compactness} (cf. \cite[Theorem~1]{pg22}). With the notation of Theorem~\ref{thm:plates-compactness}, it follows that $v(x)=y\bra{x_1,\frac{x_2}{w},x_3}$. Let $\P$ be the orthogonal projection from $L^2(\O;\R^{2\times 2}_\sym)$ onto the space of $x_3$-affine functions. As discussed in \S\ref{sec:excess-energy}, it is straight forward to check that 
	$\P(f)(x)=\int_{-\frac12}^{\frac12} fdx_3 + x_3 \cdot 12 \int_{-\frac12}^{\frac12}x_3 f dx_3$.  
	Hence \eqref{eq:appendix-functional} reads
	\[
	\widetilde{J}^w(v) = \frac{1}{24} \cdot\frac{1}{w}\int_{(0,L)\times \bra{-\frac{w}{2},\frac{w}{2}}} Q_2\bra{\D v^\top \D\nu_v - \overline{\II}}dx' + \frac12\cdot \frac{1}{w} \int_{\O^w} Q_2\bra{(\id-\P)(\check{B}_0)}dx,
	\]
	where $\overline{\II}(x')=-12\int_{-\frac12}^{\frac12}x_3 \check{B}_0(x',x_3)dx_3$, as in Lemma~\ref{lem:Q2-bar-def}. Lastly, in order to obtain a functional in terms of $y$, another change of coordinates $x_2\mapsto w x_2$ shows that $\frac{1}{t^2}I^{t,w}(y^t)$ converges (again, in the appropriate $\Gamma$-sense) to
	\[
	J^w(y)\coloneqq
	\frac{1}{24}\int_S Q_2\bra{\II^w_y (x') - \overline{\II}(x')}dx'+
	\underbrace{ \frac12 \int_\O Q_2\bra{ (\id-\P)(\check{B}(x_1,x_2,x_3)) }dx   }_{\eqqcolon\, C_{\text{excess}}^{(0)}}
	,
	\]
	where $\II^w_y\coloneqq \D_wy^\top \D_w \nu_y$ and $\overline{\II}, C_{\text{excess}}^{(0)}$ are given in \eqref{eq:plates-gamma-convergence-reference}.
\end{proof}

\end{appendices}

\newpage
{\footnotesize
	\bibliographystyle{amsalpha}
	\bibliography{paper_refs}	

@article{gsd16,
	title={Elasticity and Fluctuations of Frustrated Nanoribbons},
	volume={116},
	ISSN={1079-7114},
	url={http://dx.doi.org/10.1103/PhysRevLett.116.258105},
	DOI={10.1103/physrevlett.116.258105},
	number={25},
	journal={Physical Review Letters},
	publisher={American Physical Society (APS)},
	author={Grossman, Doron and Sharon, Eran and Diamant, Haim},
	year={2016},
	month=jun }

@article{k1850,
	author = {Kirchhoff, G.},
	journal = {Journal für die reine und angewandte Mathematik},
	language = {ger},
	pages = {51-88},
	title = {Über das Gleichgewicht und die Bewegung einer elastischen Scheibe.},
	url = {http://eudml.org/doc/147439},
	volume = {40},
	year = {1850},
}

@article{zgds19,
	author = {Zhang, Mingming and Grossman, Doron and Danino, Dganit and Sharon, Eran},
	year = {2019},
	month = {08},
	pages = {},
	title = {Shape and fluctuations of frustrated self-assembled nano ribbons},
	volume = {10},
	journal = {Nature Communications},
	doi = {10.1038/s41467-019-11473-6}
}

@book{love27,
	title={A Treatise on the Mathematical Theory of Elasticity},
	author={Love, A.E.H.},
	year={1927},
	publisher={Cambridge University Press, Cambridge}
}

@article{ks14,
	title={A {R}iemannian approach to reduced plate, shell, and rod theories},
	volume={266},
	ISSN={0022-1236},
	url={http://dx.doi.org/10.1016/j.jfa.2013.09.003},
	DOI={10.1016/j.jfa.2013.09.003},
	number={5},
	journal={Journal of Functional Analysis},
	publisher={Elsevier BV},
	author={Kupferman, Raz and Solomon, Jake P.},
	year={2014},
	month=mar, pages={2989–3039} }

@inbook{ekm20,
	title = {Limits of Distributed Dislocations in Geometric and Constitutive Paradigms},
	ISBN = {9783030426835},
	ISSN = {1876-9896},
	url = {http://dx.doi.org/10.1007/978-3-030-42683-5_8},
	DOI = {10.1007/978-3-030-42683-5_8},
	booktitle = {Geometric Continuum Mechanics},
	publisher = {Springer International Publishing},
	author = {Epstein,  Marcelo and Kupferman,  Raz and Maor,  Cy},
	year = {2020},
	pages = {349–380}
}

@article{Maor25,
	title = {On material-uniform elastic bodies with disclinations and their homogenization},
	ISSN = {1741-3028},
	url = {http://dx.doi.org/10.1177/10812865251322412},
	DOI = {10.1177/10812865251322412},
	journal = {Mathematics and Mechanics of Solids},
	publisher = {SAGE Publications},
	author = {Maor,  Cy},
	year = {2025},
	month = apr 
}

@article{lssm21,
	title={Hierarchy of geometrical frustration in elastic ribbons: Shape-transitions and energy scaling obtained from a general asymptotic theory},
	volume={156},
	ISSN={0022-5096},
	url={http://dx.doi.org/10.1016/j.jmps.2021.104579},
	DOI={10.1016/j.jmps.2021.104579},
	journal={Journal of the Mechanics and Physics of Solids},
	publisher={Elsevier BV},
	author={Levin, Ido and Siéfert, Emmanuel and Sharon, Eran and Maor, Cy},
	year={2021},
	month=nov, pages={104579} }

@article{cmsh12,
	author = {Chen, Zi and Majidi, Carmel and Srolovitz, D.J. and Haataja, Mikko},
	year = {2012},
	month = {09},
	pages = {},
	title = {Continuum Elasticity Theory Approach for Spontaneous Bending and
	Twisting of Ribbons Induced by Mechanical Anisotropy}
}

@article{bls16,
	title = {Plates with Incompatible Prestrain},
	volume = {221},
	ISSN = {1432-0673},
	url = {http://dx.doi.org/10.1007/s00205-015-0958-7},
	DOI = {10.1007/s00205-015-0958-7},
	number = {1},
	journal = {Archive for Rational Mechanics and Analysis},
	publisher = {Springer Science and Business Media LLC},
	author = {Bhattacharya,  Kaushik and Lewicka,  Marta and Sch\"{a}ffner,  Mathias},
	year = {2016},
	month = jan,
	pages = {143–181}
}

@article{ko18,
	author = {Kohn, Robert and O’Brien, Ethan},
	year = {2018},
	month = {01},
	pages = {},
	title = {On the Bending and Twisting of Rods with Misfit},
	volume = {130},
	journal = {Journal of Elasticity},
	doi = {10.1007/s10659-017-9635-4}
}

@article{sls21,
	title = {{E}uclidean Frustrated Ribbons},
	author = {Si\'efert, Emmanuel and Levin, Ido and Sharon, Eran},
	journal = {Phys. Rev. X},
	volume = {11},
	issue = {1},
	pages = {011062},
	numpages = {11},
	year = {2021},
	month = {Mar},
	publisher = {American Physical Society},
	doi = {10.1103/PhysRevX.11.011062},
	url = {https://link.aps.org/doi/10.1103/PhysRevX.11.011062}
}

@article{fjm2002,
author = {Friesecke, Gero and James, Richard D. and Müller, Stefan},
title = {A theorem on geometric rigidity and the derivation of nonlinear plate theory from three-dimensional elasticity},
journal = {Communications on Pure and Applied Mathematics},
volume = {55},
year = {2002},
number = {11},
pages = {1461-1506},
doi = {https://doi.org/10.1002/cpa.10048},
url = {https://onlinelibrary.wiley.com/doi/abs/10.1002/cpa.10048},
eprint = {https://onlinelibrary.wiley.com/doi/pdf/10.1002/cpa.10048}
}

@article{fjm06,
	author = {Friesecke, Gero and James, Richard and Müller, Stefan},
	year = {2006},
	month = {05},
	pages = {183-236},
	title = {A Hierarchy of Plate Models Derived from Nonlinear Elasticity by Gamma-Convergence},
	volume = {180},
	journal = {Archive for Rational Mechanics and Analysis},
	doi = {10.1007/s00205-005-0400-7}
}

@misc{mm25,
	title={Rigorous analysis of shape transitions in frustrated elastic ribbons}, 
	author={Cy Maor and Maria Giovanna Mora},
	year={2025},
	eprint={2503.11779},
	archivePrefix={arXiv},
	primaryClass={math.AP},
	url={https://arxiv.org/abs/2503.11779}, 
}

@article{all19,
	author = {{Agostiniani, Virginia} and {Lucantonio, Alessandro} and {Lučić, Danka}},
	title = {Heterogeneous elastic plates with in-plane modulation of the target curvature and applications to thin gel sheets},
	DOI= "10.1051/cocv/2018046",
	url= "https://doi.org/10.1051/cocv/2018046",
	journal = {ESAIM: COCV},
	year = 2019,
	volume = 25,
	pages = "24",
}

@article{ll20,
	author = {Lewicka, Marta and Lučić, Danka},
	title = {Dimension Reduction for Thin Films with Transversally Varying Prestrain: Oscillatory and Nonoscillatory Cases},
	journal = {Communications on Pure and Applied Mathematics},
	volume = {73},
	number = {9},
	pages = {1880-1932},
	doi = {https://doi.org/10.1002/cpa.21871},
	url = {https://onlinelibrary.wiley.com/doi/abs/10.1002/cpa.21871},
	eprint = {https://onlinelibrary.wiley.com/doi/pdf/10.1002/cpa.21871},
	year = {2020}
}

@article{ms19,
	author = {Maor, Cy and Shachar, Asaf},
	year = {2019},
	month = {02},
	pages = {},
	title = {On the Role of Curvature in the Elastic Energy of Non-{E}uclidean Thin Bodies},
	volume = {134},
	journal = {Journal of Elasticity},
	doi = {10.1007/s10659-018-9686-1}
}

@article{
	aeks11,
	author = {Shahaf Armon  and Efi Efrati  and Raz Kupferman  and Eran Sharon },
	title = {Geometry and Mechanics in the Opening of Chiral Seed Pods},
	journal = {Science},
	volume = {333},
	number = {6050},
	pages = {1726-1730},
	year = {2011},
	doi = {10.1126/science.1203874},
	URL = {https://www.science.org/doi/abs/10.1126/science.1203874},
	eprint = {https://www.science.org/doi/pdf/10.1126/science.1203874}
}

@article{ldr96,
	title={The membrane shell model in nonlinear elasticity: a variational asymptotic derivation},
	author={Le Dret, Herv{\'e} and Raoult, Annie},
	journal={Journal of Nonlinear Science},
	volume={6},
	pages={59--84},
	year={1996},
	publisher={Springer}
}

@article{mm03,
	author = {Mora, Maria Giovanna and Müller, Stefan},
	year = {2003},
	month = {11},
	pages = {287-305},
	title = {Derivation of the nonlinear bending-torsion theory for inextensible rods by $\Gamma$ -convergence},
	volume = {18},
	journal = {Calculus of Variations and Partial Differential Equations},
	doi = {10.1007/s00526-003-0204-2}
}

@article{lewpak11,
	author = {Lewicka, Marta and Pakzad, Mohammad Reza},
	title = {Scaling laws for non-{E}uclidean plates and the {$W^{2,2}$} isometric immersions of {R}iemannian metrics},
	journal = {ESAIM: Control, Optimisation and Calculus of Variations},
	pages = {1158--1173},
	publisher = {EDP-Sciences},
	volume = {17},
	number = {4},
	year = {2011},
	doi = {10.1051/cocv/2010039},
	mrnumber = {2859870},
	zbl = {1300.74028},
	language = {en},
	url = {http://www.numdam.org/articles/10.1051/cocv/2010039/}
}

@article{sch07,
	title = {Plate theory for stressed heterogeneous multilayers of finite bending energy},
	journal = {Journal de Mathématiques Pures et Appliquées},
	volume = {88},
	number = {1},
	pages = {107-122},
	year = {2007},
	issn = {0021-7824},
	doi = {https://doi.org/10.1016/j.matpur.2007.04.011},
	url = {https://www.sciencedirect.com/science/article/pii/S0021782407000487},
	author = {Bernd Schmidt}
}

@article{fmp12,
	author = {Freddi, Lorenzo and Mora, Maria Giovanna and Paroni, Roberto},
title = {Nonlinear Thin-Walled Beams With a Rectangular Cross-Section - Part {I}},
	journal = {Mathematical Models and Methods in Applied Sciences},
	volume = {22},
	number = {03},
	pages = {1150016},
	year = {2012},
	doi = {10.1142/S0218202511500163},	
	URL = { 	
	https://doi.org/10.1142/S0218202511500163
	},
	eprint = { 	
	https://doi.org/10.1142/S0218202511500163
	}
}

@article{fhmp16b,
	author = {Freddi, Lorenzo and Hornung, Peter and Mora, Maria Giovanna and Paroni, Roberto},
	title = {A Variational Model for Anisotropic and Naturally Twisted Ribbons},
	journal = {SIAM Journal on Mathematical Analysis},
	volume = {48},
	number = {6},
	pages = {3883-3906},
	year = {2016},
	doi = {10.1137/16M1074862},
		URL = { https://doi.org/10.1137/16M1074862	},
	eprint = { 	https://doi.org/10.1137/16M1074862	
	}
}

@book{cynotes,
	author        = {Maor, Cy},
	title         = {Compact Course "Non-{E}uclidean Elasticity: thin bodies and material defects"},
	month         = {September},
	year          = {2023},
	publisher={Technische Universität Dresden}
}

@book{lewicka-prestressed,
	author        = {Lewicka, Marta},
	title         = {Calculus of Variations on Thin Prestressed Films},
	subtitle = {Asymptotic Methods in Elasticity},
	year          = {2023},
	publisher={Birkhäuser Cham},
	doi={10.1007/978-3-031-17495-7}
}

@book{antman95,
	author        = {Antman, Stuart S.},
	title         = {Nonlinear Problems of Elasticity},
	year          = {1995},
	publisher={Springer New York, NY},
	doi={10.1007/978-1-4757-4147-6}
}

@book{ciarlet00,
	author        = {Ciarlet, P.G.},
	title         = {Mathematical Elasticity, {Vol. I-III}},
	year          = {2000},
	publisher={North-Holland, Amsterdam}
}

@article{mm08, 
	author={Mora, M. G. and Müller, S.},
	title={Convergence of equilibria of three-dimensional thin elastic beams}, 
	volume={138},
	DOI={10.1017/S0308210506001120},
	number={4},
	journal={Proceedings of the Royal Society of Edinburgh: Section A Mathematics},
	year={2008}, 
	pages={873–896}}

@article{bff00,
	ISSN = {00222518, 19435258},
	URL = {http://www.jstor.org/stable/24901104},
	author = {A. Braides and I. Fonseca and G. Francfort},
	journal = {Indiana University Mathematics Journal},
	number = {4},
	pages = {1367--1404},
	publisher = {Indiana University Mathematics Department},
	title = {{3D-2D} Asymptotic Analysis for Inhomogeneous Thin Films},
	urldate = {2025-05-02},
	volume = {49},
	year = {2000}
}

@article{lrr17,
	author = {Lewicka, Marta and Raoult, Annie and Ricciotti, Diego},
	title = {Plates with incompatible prestrain of high order},
	journal = {Annales de l'I.H.P. Analyse non lin\'eaire},
	pages = {1883--1912},
	publisher = {Elsevier},
	volume = {34},
	number = {7},
	year = {2017},
	doi = {10.1016/j.anihpc.2017.01.003},
	mrnumber = {3724760},
	zbl = {1457.74121},
	language = {en},
	url = {https://www.numdam.org/articles/10.1016/j.anihpc.2017.01.003/}
}

@article{bnps22,
	title={A Homogenized Bending Theory for Prestrained Plates},
	volume={33},
	ISSN={1432-1467},
	url={http://dx.doi.org/10.1007/s00332-022-09869-8},
	DOI={10.1007/s00332-022-09869-8},
	number={1},
	journal={Journal of Nonlinear Science},
	publisher={Springer Science and Business Media LLC},
	author={Böhnlein, Klaus and Neukamm, Stefan and Padilla-Garza, David and Sander, Oliver},
	year={2022},
	month=dec }

@article{crs17,
	title = {On global and local minimizers of prestrained thin elastic rods},
	volume = {56},
	ISSN = {1432-0835},
	url = {http://dx.doi.org/10.1007/s00526-017-1197-6},
	DOI = {10.1007/s00526-017-1197-6},
	number = {4},
	journal = {Calculus of Variations and Partial Differential Equations},
	publisher = {Springer Science and Business Media LLC},
	author = {Cicalese,  Marco and Ruf,  Matthias and Solombrino,  Francesco},
	year = {2017},
	month = jul 
}

@article{pg22,
	title = {Dimension reduction through gamma convergence for general prestrained thin elastic sheets},
	volume = {61},
	ISSN = {1432-0835},
	url = {http://dx.doi.org/10.1007/s00526-022-02262-z},
	DOI = {10.1007/s00526-022-02262-z},
	number = {5},
	journal = {Calculus of Variations and Partial Differential Equations},
	publisher = {Springer Science and Business Media LLC},
	author = {Padilla-Garza,  David},
	year = {2022},
	month = jul 
}
}

\end{document}